\documentclass[hidelinks,onefignum,onetabnum]{siamart251216}
\usepackage{lipsum}
\usepackage{amsfonts}
\usepackage{amssymb}
\usepackage{graphicx}
\usepackage{mathrsfs}
\usepackage{epstopdf}
\ifpdf
  \DeclareGraphicsExtensions{.eps,.pdf,.png,.jpg}
\else
  \DeclareGraphicsExtensions{.eps}
\fi

\newsiamremark{remark}{Remark}
\newsiamremark{hypothesis}{Hypothesis}
\crefname{hypothesis}{Hypothesis}{Hypotheses}
\newsiamthm{claim}{Claim}
\newsiamremark{fact}{Fact}
\crefname{fact}{Fact}{Facts}

\headers{Low-cost Sylvester solvers}{J. Saha, J. Yang, and P. Yin}

\title{Low-cost non-iterative Sylvester solvers for dynamical low-rank approximation of reaction-diffusion equations}
\author{Joy Saha\thanks{Computational Science Program, The University of Texas at El Paso, El Paso, TX 79968, USA (\email{jsaha@miners.utep.edu}).}
\and Jun Yang\thanks{Department of Mathematics, Southern University of Science and Technology, Shenzhen 518000, P.R.China (\email{yangjun@sustech.edu.cn}).}
\and Peimeng Yin\thanks{Corresponding author. Department of Mathematical Sciences, The University of Texas at El Paso, El Paso, TX 79968, USA; Computational Science Program, The University of Texas at El Paso, El Paso, TX 79968, USA (\email{pyin@utep.edu}).}} 

\usepackage{amsopn}
\newcommand{\R}{\mathbb{R}}

\usepackage{booktabs}
\usepackage{algorithm}
\usepackage{algpseudocode}
\newsiamthm{assumption}{Assumption}
\crefname{assumption}{Assumption}{Assumptions}
\newsiamthm{example}{Example}

\ifpdf
\hypersetup{
  pdftitle={An Example Article},
  pdfauthor={D. Doe, P. T. Frank, and J. E. Smith}
}
\fi

\ifpdf
  \DeclareGraphicsExtensions{.eps,.pdf,.png,.jpg}
\else
  \DeclareGraphicsExtensions{.eps}
\fi
\usepackage[margin=1in]{geometry}

\usepackage{graphicx}
\usepackage{subcaption}
\usepackage{multirow}
\usepackage{booktabs}
\usepackage{adjustbox}

\usepackage{algorithm}
\usepackage{algpseudocode}
\let\FOR\For
\let\ENDFOR\EndFor
\let\REQUIRE\Require
\let\ENSURE\Ensure
\let\STATE\State

\let\RETURN\Return

\newcommand{\algcomment}[1]{%
  \item[]\hskip-\algorithmicindent
  {\small\itshape // #1}%
}

\renewcommand{\proofbox}{}

\begin{document}

\maketitle 

\begin{abstract}
In this work, we develop low-cost non-iterative Sylvester solvers for dynamical low-rank approximation of reaction-diffusion equations and incorporate them into a basis-update \& Galerkin mass-lumped finite element method (\mbox{BUG-MLFEM}). By incorporating mass lumping via Gauss-Lobatto quadrature, the spatial discretization yields a matrix differential equation equipped with diagonal positive-definite mass matrices. We establish the weighted dynamical low-rank approximation (WDLRA) within the Galerkin framework of the MLFEM system, proving the theoretical equivalence between the BUG integrator and MLFEM on low-rank subspaces, which guarantees the discrete well-posedness of the implicit BUG steps. To overcome the severe computational bottleneck of standard $O(m^3+n^3)$ Sylvester solvers, we develop tailored, low-cost non-iterative solvers that exploit the inherent sparse-by-small algebraic structure. Under Dirichlet and Neumann boundary conditions (BCs), the large coefficient matrices are banded, enabling efficient diagonalization of the small $r \times r$ factor matrix and reducing the per-step computational complexity to $O((m+n)r^2)$ and storage to $O((m+n)r)$. For periodic BCs, we combine this diagonalization strategy with a rank-2 Sherman-Morrison-Woodbury update to handle wrap-around corner entries at negligible additional cost. Numerical experiments demonstrate first-order temporal and $(k+1)$th-order spatial convergence of the proposed BUG-MLFEM in the discrete $L^2$ norm, where $k$ denotes the polynomial degree of the finite element space, together with significant computational speedups over existing low-rank integrators.
\end{abstract}

% REQUIRED

% REQUIRED

\section{Introduction}
Reaction-diffusion equations serve as fundamental mathematical models across physical, chemical, and biological sciences, describing multi-species interactions governed by local non-linear reaction kinetics and global spatial diffusion \cite{murray2003mathematical}.  Numerical discretizations such as finite difference, finite element (FEM), finite volume, and spectral methods have been extensively developed for these systems \cite{thomee2006galerkin}. Among them, finite element methods provide superior geometric flexibility and high-order accuracy. 
However, fine spatial discretizations, particularly when paired with high-order finite elements, give rise to massive matrix differential equations, rendering standard full-rank time-marching integrators practically unfeasible.

Discretizing high-dimensional partial differential equations (PDEs) on fine spatial grids typically yields massive algebraic systems whose direct numerical integration is computationally prohibitive in terms of both execution time and memory footprint; see, e.g., \cite{grasedyck2013literature}. 
Dynamical low-rank approximation (DLRA) \cite{koch2007dynamical} has emerged as an indispensable model-reduction paradigm to alleviate this severe computational bottleneck by exploiting the intrinsic low-dimensional manifold structure of evolving solutions. Rather than marching the full-order high-dimensional state, DLRA projects the governing dynamics onto a time-dependent low-rank manifold, advancing only the underlying low-dimensional factor matrices. This substantially reduces both the active degrees of freedom and memory requirements, establishing DLRA as a powerful approach for large-scale matrix differential equations across a wide range of applications, including kinetic equations \cite{yin2025towards, einkemmer2025review}, wave propagation \cite{hochbruck2023rank, zhao2025low}, uncertainty quantification \cite{musharbash2020symplectic, kusch2022dynamical}, machine learning \cite{schotthofer2022low, schmidt2023rank}, and diffusion-type equations \cite{nakao2025reduced, el2025sylvester, meng2025preconditioning, yang2026second}.

DLRA relies on robust splitting-based time integrators, among which the projector-splitting (PS) integrator \cite{lubich2014projector, kieri2016discretized} and the Basis-Update \& Galerkin (BUG) integrators \cite{ceruti2022unconventional, ceruti2024robust, ceruti2022rank} are the most prominent. While the PS integrator performs remarkably well across various settings, its intermediate backward-in-time evolution step can induce severe numerical instabilities, particularly when applied to parabolic problems or systems with high stiffness. In contrast, BUG integrators overcome these limitations by updating the basis matrices in parallel and evolving the low-rank factors exclusively forward in time. This structural feature not only simplifies extensions to structure-preserving discretizations \cite{ceruti2022rank} but also accommodates implicit sub-steps to handle stiff dynamics robustly. A key milestone in this framework is the augmented BUG integrator \cite{ceruti2023rank}, which offers enhanced flexibility and rank adaptivity. 

More recently, DLRA has gained increasing traction in solving diffusion-type problems \cite{nakao2025reduced, el2025sylvester, meng2025preconditioning, yang2026second}, with the BUG framework taking a central role. In BUG integrator \cite{ceruti2022unconventional}, the rank-constrained evolution decouples into independent $K$-, $L$-, and $S$-steps. However, spatial diffusion operators are inherently stiff. To overcome severe Courant--Friedrichs--Lewy (CFL) stability limits, implicit time discretizations are imperative. Consequently, the implicit update in both $K$- and $L$-steps requires solving large-scale Sylvester-type matrix equations at each time step. Standard dense solvers, such as the Bartels--Stewart algorithm \cite{bartels1972algorithm}, incur an $O(m^3+n^3)$ computational cost and $O(m^2+n^2)$ memory footprint per time step, where $m$ and $n$ denote the spatial resolutions along the respective coordinate directions. For high-resolution spatial grids, these dense algebraic operations become the dominant computational bottleneck, largely canceling out the speedup achieved by the low-rank formulation.

To address this challenge, several low-rank implicit integrators have been explored. The Reduced Augmentation Implicit Low-rank (RAIL) method \cite{nakao2025reduced} and its 3D extension \cite{nakao2026low} reduce the memory footprint and computational cost to ${O}(m^2r)$ by leveraging low-rank structures. For $ d$-dimensional advection-diffusion problems, preconditioned adaptive-rank Krylov/GMRES solvers were introduced in \cite{el2025sylvester}, achieving efficient solves with ${O}(m r^2 + r^{d+1})$ complexity under second-order finite difference stencils. In addition, a preconditioned low-rank generalized minimal residual (GMRES) method for implicit discretizations of matrix differential equations was proposed in \cite{meng2025preconditioning}, where a second-order finite difference discretization is employed in space. Exact low-rank solvers for the linear diffusion operator in Allen-Cahn equations were also developed in \cite{yang2026second}, achieving $O((m^2+n^2)r)$ computational complexity per time step.

Despite these notable advances, combining implicit low-rank integrators with high-order finite element methods presents non-trivial computational and analytical challenges. First, non-diagonal mass matrices originating from standard high-order finite elements introduce spatial couplings, which can restrict the direct applicability of fast banded Sylvester solvers. Second, iterative solvers often rely on problem-specific preconditioning, and their efficiency can be sensitive to solver convergence or temporary rank inflation during rapid transient dynamics. Finally, establishing rigorous discrete well-posedness and $L^2$-stability remains a delicate question that warrants further formal analysis.

To overcome these fundamental mathematical and algorithmic challenges, this paper presents a BUG-MLFEM for 2D reaction-diffusion equations under Dirichlet, Neumann, and periodic BCs. By incorporating Gauss-Lobatto numerical quadrature into tensor-product finite element spaces, we obtain strictly diagonal mass matrices, preserving optimal $O(h^{k+1})$ spatial accuracy while completely eliminating mass coupling \cite{li2020arbitrarily}. Our main contributions are as follows.
\begin{itemize}
    \item[(i)] \textbf{Theoretical rigor and well-posedness:} 
    By recasting the mass-lumped finite element system into the WDLRA framework, we demonstrate the exact equivalence between the BUG integrator and MLFEM on low-rank subspaces, thereby ensuring the rigorous well-posedness of the discrete subproblems at each time step.
    \item[(ii)] \textbf{High-order spatial discretization:} By incorporating Gauss--Lobatto quadrature, the spatial mass matrices become strictly diagonal and positive definite, which eliminates inter-element mass coupling while retaining the optimal $O(h^{k+1})$ spatial error for $P_k$ finite elements. Numerical experiments confirm first-order convergence in time and $(k+1)$th-order convergence in space in the $L^2$-norm.
    \item[(iii)] \textbf{Low-cost non-iterative Sylvester solvers:} For Dirichlet and Neumann BCs, the $K/L$-steps naturally reduce to sparse-by-small Sylvester equations $AX + XB = C$, where $A$ is large and banded, and $B \in \mathbb{R}^{r \times r}$ is small and symmetric. By diagonalizing the small matrix $B$, the system is completely decoupled into $r$ independent, banded linear equations. This allows direct solving without any Krylov iterations, preconditioners, achieving optimal $O((m+n)r^2)$ computational complexity.

    \item[(iv)] \textbf{Handling periodic BCs via Woodbury correction:} For periodic BCs, corner wrap-around entries in stiffness matrices are handled using a rank-2 Sherman-Morrison-Woodbury update, preserving purely banded solves at negligible additional cost.

    \item[(v)] \textbf{Optimal computational and storage complexity:} 
    The proposed solver reduces the computational complexity per time step from $O(m^3+n^3)$ (or $O((m^2+n^2)r)$) down to $O((m+n)r^2)$ and the memory requirement down to $O((m+n)r)$. Numerical tests show up to an $81\times$ computational speedup over full-rank integrators on fine grids ($N=160$).
\end{itemize}

The rest of this paper is organized as follows. In \Cref{sec:mde}, we formulate the model problem, present its matrix differential equation, and introduce the full-rank MLFEM. In \Cref{DLRA_for_the_Diff_Reac_Eqn}, we develop the WDLRA, establish the equivalence between the BUG integrator and the MLFEM, and prove the well-posedness of the resulting BUG scheme. \Cref{sec:solvers} presents low-cost, non-iterative Sylvester solvers for different boundary conditions and analyzes their computational complexity. \Cref{sec:experiments} presents numerical experiments on convergence rates and computational efficiency. Conclusions are given in \Cref{sec:conclusion}. Finally, augmented BUG integrators enabled by the proposed non-iterative solvers are provided in \Cref{abug}.

\section{Model problem and matrix differential equation}\label{sec:mde}

In this paper, we consider the two-dimensional reaction-diffusion equation on the rectangular domain $\Omega=[a,b]\times[c,d]\subset\mathbb{R}^{2}$:
\begin{equation}
w_{t}=\Delta w+\alpha w, \quad (x,y)\in\Omega, \; t\in(0,T],
\label{diffusion}
\end{equation}
with the initial condition $w(x,y,0)=w_{0}(x,y),$
where $\alpha\le0$ is the reaction coefficient and $w_{0}$ is given. We consider one of the following BCs on $\partial\Omega$:
\begin{enumerate}
    \item[(i)] Homogeneous Dirichlet: $w=0$.
    \item[(ii)] Homogeneous Neumann: $\frac{\partial w}{\partial n}=0$, where $n$ is the outward unit normal.
    \item[(iii)] Periodic: $w(a,y,t)=w(b,y,t)$, $w_{x}(a,y,t)=w_{x}(b,y,t)$, $w(x,c,t)=w(x,d,t)$, and $w_{y}(x,c,t)=w_{y}(x,d,t)$.
\end{enumerate}

%==============================================================================================
\subsection{Mass-lumped finite element space}
%==============================================================================================

We formulate the matrix differential equation associated with problem \eqref{diffusion}.
We partition the rectangular domain
\(
\Omega=[a,b]\times[c,d]
\)
into \(M\) and \(N\) uniform elements in the \(x\)- and
\(y\)-directions, respectively. Let
\[
a=\xi_0<\xi_1<\cdots<\xi_M=b,
\qquad
c=\eta_0<\eta_1<\cdots<\eta_N=d
\]
be the mesh nodes in each direction, with mesh sizes $h_x=\frac{b-a}{M}$, $h_y=\frac{d-c}{N}.$
The corresponding one-dimensional subintervals are $I_i^x=[\xi_{i-1},\xi_i]$, $i=1,\ldots,M,$ and $I_j^y=[\eta_{j-1},\eta_j]$, $j=1,\ldots,N$, respectively.
Therefore, the domain is partitioned into rectangular cells as 
\(
\Omega
=
\bigcup_{i=1}^{M}\bigcup_{j=1}^{N}
\left(I_i^x\times I_j^y\right).
\)

In each coordinate direction, we use $k+1$ Gauss--Lobatto
quadrature points on every mesh interval. More precisely, for
$I_i^x=[\xi_{i-1},\xi_i]$, let
\(
x_{(i-1)k+\ell},\
\omega_\ell^x,\
\ell=1,\ldots,k+1,
\)
denote the Gauss--Lobatto quadrature points and the corresponding
quadrature weights, respectively. Thus,
\[
x_{(i-1)k+1}=\xi_{i-1},
\qquad
x_{ik+1}=\xi_i.
\]
Similarly, for $I_j^y=[\eta_{j-1},\eta_j]$, let
\(
y_{(j-1)k+\ell},\
\omega_\ell^y,\
\ell=1,\ldots,k+1,
\)
denote the corresponding Gauss--Lobatto quadrature points and
quadrature weights, with
\[
y_{(j-1)k+1}=\eta_{j-1},
\qquad
y_{jk+1}=\eta_j.
\]

The global quadrature weights are obtained by assembling the local
Gauss--Lobatto weights at the coinciding mesh nodes. More precisely,
for a direction with $L$ elements, we define the assembled weights
\[
\bar{\omega}_{(i-1)k+\ell}
=
\begin{cases}
\omega_1, & i=1,\ \ell=1,\\[1mm]
\omega_\ell, & \ell=2,\ldots,k,\\[1mm]
\omega_{k+1}+\omega_1,
& i=1,\ldots,L-1,\ \ell=k+1,\\[1mm]
\omega_{k+1}, & i=L,\ \ell=k+1.
\end{cases}
\]
Here, $\omega_\ell=\omega_\ell^x$ and $L=M$ in the $x$-direction,
while $\omega_\ell=\omega_\ell^y$ and $L=N$ in the $y$-direction,
for $\ell=1,\ldots,k+1$.
Since the Gauss--Lobatto endpoint weights satisfy
$\omega_1=\omega_{k+1}$, the interior mesh interfaces receive the
combined weight $2\omega_{k+1}$, while each physical boundary node
receives only one endpoint contribution. Applying this definition in
the $x$- and $y$-directions, respectively, gives
% \[
% \{\bar{\omega}_p^x\}_{p=1}^{Mk+1}
% \quad\text{and}\quad
% \{\bar{\omega}_q^y\}_{q=1}^{Nk+1}.
% \]
the assembled global Gauss--Lobatto quadrature nodes and weights
\[
\{(x_p,\bar{\omega}_p^x)\}_{p=1}^{Mk+1},
\qquad
\{(y_q,\bar{\omega}_q^y)\}_{q=1}^{Nk+1},
\]
where the weights at coinciding element endpoints are summed.

Using the Gauss--Lobatto quadrature, we define the mass-lumped inner
products in one and two dimensions, respectively, by
\begin{equation}\label{1dinner}
(f,g)_h
=
\sum_p \bar{\omega}_p f(x_p)g(x_p)
\end{equation}
and
\begin{equation}\label{mass_lumped_inner_product}
(f,g)_h
=
\sum_{p=1}^{Mk+1}
\sum_{q=1}^{Nk+1}
\bar{\omega}_p^x\bar{\omega}_q^y
f(x_p,y_q)g(x_p,y_q).
\end{equation}
The associated mass-lumped $L^2$ norm is $\|f\|_h=(f,f)_h^{1/2}$.

Define the one-dimensional finite element spaces
\begin{align}
V_{x,h}
&=
\left\{
\varphi\in V_x:
\left.\varphi\right|_{I_i^x}\in P_k,
\quad i=1,\ldots,M
\right\},
\\
V_{y,h}
&=
\left\{
\psi\in V_y:
\left.\psi\right|_{I_j^y}\in P_k,
\quad j=1,\ldots,N
\right\},
\end{align}
where $P_k$ denotes the space of polynomials of degree at most $k$.
The spaces $V_x$ and $V_y$ incorporate the prescribed boundary
conditions. Specifically, for homogeneous Dirichlet, homogeneous
Neumann, and periodic BCs, respectively, we set
\[
\begin{array}{c|c|c}
\text{Boundary condition}
& V_x & V_y \\ \hline
\text{Homogeneous Dirichlet}
& H_0^1(I^x) & H_0^1(I^y)\\
\text{Homogeneous Neumann}
& H^1(I^x) & H^1(I^y)\\
\text{Periodic}
& H_{\mathrm{per}}^1(I^x) & H_{\mathrm{per}}^1(I^y).
\end{array}
\]

The two-dimensional space is defined by
\(
V=V_x\otimes V_y,
\)
and the corresponding tensor-product finite element space is
\begin{equation}\label{FEMspace}
V_h
=
V_{x,h}\otimes V_{y,h}
=
\left\{
v\in V:
\left.v\right|_K\in Q_k,
\quad \forall K\in\mathcal{T}_h
\right\},
\end{equation}
where
\(
\mathcal{T}_h
=
\left\{
K_{ij}=I_i^x\times I_j^y:
1\leq i\leq M,\;
1\leq j\leq N
\right\}
\)
is the rectangular mesh, and
$Q_k=P_k\otimes P_k$
is the tensor-product polynomial space consisting of polynomials of
degree at most $k$ in each coordinate.

Depending on the boundary conditions, we denote
\begin{equation}
m = Mk + \delta_x, \qquad
n = Nk + \delta_y,
\end{equation}
where
\(
(\delta_x,\delta_y)=
\begin{cases}
(-1,-1), & \text{homogeneous Dirichlet},\\
(+1,+1), & \text{homogeneous Neumann},\\
(0,0), & \text{periodic}.
\end{cases}
\)

Let
\(\{\varphi_i(x)\}_{i=1}^m\subset V_{x,h}^k\) and
\(\{\psi_j(y)\}_{j=1}^n\subset V_{y,h}^k\) denote the nodal basis
functions satisfying
\begin{equation}
\label{eq:noal_func}
  \varphi_i(x_p)=\delta_{ip},
  \qquad
  \psi_j(y_q)=\delta_{jq}.
\end{equation}
The corresponding vector-valued basis functions are
\begin{equation}
  \Phi(x)
  =
  [\varphi_1(x),\ldots,\varphi_m(x)]^\top,
  \qquad
  \Psi(y)
  =
  [\psi_1(y),\ldots,\psi_n(y)]^\top.
\end{equation}

The one-dimensional lumped mass matrices are defined by
\begin{equation}
\label{eq:M_x_M_y}
  M_x
  =
  \bigl((\varphi_i,\varphi_p)_h^x\bigr)_{i,p=1}^m,
  \qquad
  M_y
  =
  \bigl((\psi_j,\psi_q)_h^y\bigr)_{j,q=1}^n,
\end{equation}
where \((\cdot,\cdot)_h^x\) and \((\cdot,\cdot)_h^y\) denote the
one-dimensional mass-lumped inner products in the \(x\)- and
\(y\)-directions, respectively. These matrices are diagonal and postitive definite as a
consequence of the Gauss--Lobatto mass lumping.

The corresponding one-dimensional stiffness matrices are
\begin{equation}\label{AxAy}
A_x
=
\bigl((\partial_x\varphi_i,\partial_x\varphi_p)_h\bigr)_{i,p=1}^m,
\qquad
A_y
=
\bigl((\partial_y\psi_j,\partial_y\psi_q)_h\bigr)_{j,q=1}^n.
\end{equation}
Both $A_x$ and $A_y$ are sparse for all three types of boundary
conditions, while their bandwidth depends on the boundary conditions
and the polynomial degree.

%==============================================================================================
\subsection{Semi-discrete formulation and matrix differential equation}
%==============================================================================================

We seek a semi-discrete finite element solution \(w_h(t)\in V_h\) in the
form
\begin{equation}\label{masssol}
  w_h(x,y,t)=\sum_{i=1}^m\sum_{j=1}^nW_{ij}\phi_i(x)\psi_j(y)=\Phi(x)^\top W(t)\Psi(y),
\end{equation}
where \(W(t)\in\mathbb{R}^{m\times n}\) is the time-dependent coefficient
matrix.
The semi-discrete MLFEM is to find
\(w_h(t)\in V_h\) such that
\begin{equation}
  (w_{h,t},v_h)_h
  =
  -(\nabla w_h,\nabla v_h)_h
  +\alpha(w_h,v_h)_h,
  \qquad
  \forall v_h\in V_h,
  \label{eq:2d-semidiscrete}
\end{equation}
with the initial condition 
\begin{equation}
  (w_h(0),v_h)_h
  =
  (w_0,v_h)_h,
  \qquad
  \forall v_h\in V_h.
  %\label{eq:2d-initial}
\end{equation}

For \(z_h,w_h\in V_h\), let \(Z,W\in\mathbb{R}^{m\times n}\) denote their
coefficient matrices with
\begin{equation}
\label{eq:func_zh}
  z_h=\Phi^\top Z\Psi,
  \qquad
  w_h=\Phi^\top W\Psi.
\end{equation}
We introduce the \(M\)-weighted Frobenius inner product and norm
\begin{equation}
  (Z,W)_M
  :=
  (M_x Z M_y,W)_F, \quad \|Z\|_M^2 = (Z,Z)_M,
  \label{eq:M-inner-product}
\end{equation}
where \((\cdot,\cdot)_F\) denotes the Frobenius inner product. 
%We also define the one-dimensional discrete diffusion operators by
% \begin{equation}
%   L_x=-M_x^{-1}A_x,
%   \qquad
%   L_y=-M_y^{-1}A_y.
%   \label{eq:1d-diffusion-operators}
% \end{equation}

We recall the following matrix representation of inner products.
\begin{lemma}
\label{lem:inner-products} \cite{yang2026second}
For \(z_h,w_h\in V_h\) with coefficient matrices
\(Z,W\in\mathbb{R}^{m\times n}\), the mass-lumped inner products satisfy
\begin{align}
  (z_h,w_h)_h
  &=
  (Z,W)_M,
  \label{eq:ip-mass}
  \\
  (\nabla z_h,\nabla w_h)_h
  &=
  -(L_xZ+ZL_y^\top,W)_M,
  \label{eq:ip-stiff}
\end{align}
where the matrices
\begin{equation}
  L_x=-M_x^{-1}A_x,
  \qquad
  L_y=-M_y^{-1}A_y.
  \label{eq:1d-diffusion-operators}
\end{equation}
\end{lemma}

Substituting \eqref{masssol} into
\eqref{eq:2d-semidiscrete} and applying \Cref{lem:inner-products}
yields the matrix variational problem: find
\(W(t)\in\mathbb{R}^{m\times n}\) such that %for $\forall V\in\mathbb{R}^{m\times n}$, it holds
\begin{equation}
  (\partial_tW(t),V)_M
  =
  (L_xW(t)+W(t)L_y^\top,V)_M
  +\alpha(W(t),V)_M, \ \  \forall V\in\mathbb{R}^{m\times n}
  \label{eq:matrix-variational}
\end{equation}
with initial matrix \(W^0\) obtained by the mass-lumped projection of
\(w_0\) onto \(V_h\).

Since the \(M\)-weighted Frobenius inner product is positive definite,
\eqref{eq:matrix-variational} is equivalent to the following matrix
differential equation:
\begin{equation}
  \label{eq:matODE}
  \partial_tW(t)
  =
  \mathscr{L}(W(t)),
  \qquad
  W(0)=W^0,
\end{equation}
where the matrix reaction-diffusion operator
\(\mathscr{L}:\mathbb{R}^{m\times n}\to\mathbb{R}^{m\times n}\) is defined by
\begin{equation}
  \label{eq:LxLy}
  \mathscr{L}(W)
  =
  L_xW+WL_y^\top+\alpha W.
\end{equation}
Thus, the two-dimensional mass-lumped finite element discretization \eqref{eq:2d-semidiscrete} is
equivalent to the linear matrix differential equation
\eqref{eq:matODE}.

\subsection{FR-MLFEM with conventional Sylvester solver}
We consider the fully discrete scheme corresponding to \eqref{eq:2d-semidiscrete}, which we refer to as the full-rank MLFEM (FR-MLFEM). For the temporal discretization, we employ the backward Euler method with time step size $\Delta t>0$ and set $t^\mathfrak{n}=\mathfrak{n}\Delta t$. 
For each $\mathfrak{n} \ge 0$, let $w_h^\mathfrak{n}:= w_h(x,y,t^\mathfrak{n}) \in V_h$.
For simplicity, we denote
\begin{equation}
    D_t v^{\mathfrak{n}+1} = \frac{v^{\mathfrak{n}+1}-v^{\mathfrak{n}}}{\Delta t},
\end{equation}
where $v$ can be any function (or matrix in the later sections).
Given $w_h^\mathfrak{n}\in V_h$, the first-order implicit FR-MLFEM seeks $w_h^{\mathfrak{n}+1}\in V_h$ such that,
\begin{equation}
\label{eq:fd_weak_dr}
\left(
D_t w_h^{\mathfrak{n}+1},v_h
\right)_h
+
\left(
\nabla w_h^{\mathfrak{n}+1},\nabla v_h
\right)_h
-
\alpha
\left(
w_h^{\mathfrak{n}+1},v_h
\right)_h
=0, \quad v_h\in V_h.
\end{equation}

The well-posedness and stability of the the FR-MLFEM scheme \eqref{eq:fd_weak_dr} are summarized in the following result.

\begin{proposition}[Well-Posedness and $L^2$-Stability]
\label{lem:wellposed_full_rank}
For any $\Delta t>0$, let $w_h^\mathfrak{n}\in V_h$ be given. The FR-MLFEM scheme \eqref{eq:fd_weak_dr} admits a unique solution $w_h^{\mathfrak{n}+1}\in V_h$. Moreover, the solution satisfies the $L^2$-stability estimate
\begin{equation}
\label{eq:L2_stability_full_rank}
\|w_h^{\mathfrak{n}+1}\|_{h}
\leq
\|w_h^\mathfrak{n}\|_{h}.
\end{equation}
\end{proposition}

The matrix equation of the first-order FR-MLFEM \eqref{eq:fd_weak_dr} is to solve for
$W^{\mathfrak{n}+1}\in\mathbb{R}^{m\times n}$ from
\begin{equation}\label{frfemmat}
(1-\Delta t\,\alpha)W^{\mathfrak{n}+1}
-\Delta t\,L_xW^{\mathfrak{n}+1}
-\Delta t\,W^{\mathfrak{n}+1}L_y^\top
=W^{\mathfrak{n}},
\end{equation}
which can be written as the Sylvester equation
\begin{equation}\label{sylvester}
    AW^{\mathfrak{n}+1}+W^{\mathfrak{n}+1}B=W^{\mathfrak{n}},
\end{equation}
where $A:=(1-\Delta t\,\alpha)I_m-\Delta t\,L_x$, $
B:=-\Delta t\,L_y^\top.$
Since $L_x$ and $L_y$ inherit the banded sparsity of $A_x$ and $A_y$
defined in \eqref{eq:1d-diffusion-operators}, both
$A\in\mathbb{R}^{m\times m}$ and $B\in\mathbb{R}^{n\times n}$ are sparse,
but not symmetric. More specifically, they are banded under homogeneous
Dirichlet and Neumann BCs, with additional wrap-around entries under periodic BCs. Moreover, $A$ and $B$ depend only on $L_x$, $L_y$,
$\alpha$, and $\Delta t$, and remain fixed in time marching.

Solving \eqref{sylvester} using the Bartels--Stewart algorithm \cite{bartels1972algorithm}
requires $O(m^3+n^3)$ operations to compute the Schur factorizations of $A$ and $B$, followed
by $O(m^2n+mn^2)$ operations for the transformations and the solution of the resulting triangular
Sylvester equation. Since $A$ and $B$ are independent of the time step $\mathfrak{n}$, their Schur
factorizations need to be computed only once and can be reused at all subsequent time steps.
Therefore, after this one-time preprocessing cost, the computational complexity per time step is
$O(m^2n+mn^2)$.

Although the banded matrices $A$ and $B$ require only $O(m+n)$ storage, their Schur factors are generally dense \cite{golub1979hessenberg} and therefore require $O(m^2+n^2)$ storage. These factors can be computed once and reused throughout the time integration. In addition, storing the solution matrix $W^{\mathfrak{n}}\in\mathbb{R}^{m\times n}$ requires $O(mn)$ storage. Hence, the overall storage requirement is $O(m^2+n^2)$.

%==============================================================================================
\section{Weighted dynamical low-rank approximation}\label{DLRA_for_the_Diff_Reac_Eqn}
%==============================================================================================

Let $\mathcal{M}_r$ denote the manifold of rank-$r$ matrices ($r \ll \min\{m,n\}$). Throughout this work, calligraphic notation denotes low-rank matrices, while hat notation indicates their corresponding finite element approximations. Specifically, for a finite element function $z_h \in {V}_h$ with coefficient matrix $Z \in \mathbb{R}^{m \times n}$, we approximate $Z$ by a rank-$r$ matrix $\mathcal{Z} \in \mathcal{M}_r$. The dynamical low-rank finite element function associated with $\mathcal{Z}$ is denoted by $\hat{z}_h$, which lies in a subspace of ${V}_h$.

\subsection{WDLRA formulations}
The Weighted Dynamical Low-Rank Approximation (WDLRA) is a variant of DLRA that incorporates $M$-weighted norm, thereby ensuring the equivalence between its Galerkin formulation and the corresponding matrix variational problem \eqref{eq:matrix-variational}.

\begin{definition}[WDLRA]
\label{def:weighted_dlra}
The WDLRA of the matrix differential equation \eqref{eq:matODE} is defined as the rank-$r$ matrix
$\mathcal{Z}(t) \in \mathcal{M}_r$ satisfying
\begin{equation}\label{eq:weighted_dlra_dr}
\partial_t \mathcal{Z}
=
\underset{\delta Z \in T_{\mathcal{Z}}\mathcal{M}_r}{\arg\min}
\,
\|\delta Z-\mathscr{L}(\mathcal{Z})\|_M,
\end{equation}
with initial condition
\(
\mathcal{Z}(0)=\mathcal{Z}_0\in\mathcal{M}_r,
\)
where $T_{\mathcal{Z}}\mathcal{M}_r$ denotes the tangent space of $\mathcal{M}_r$ at $\mathcal{Z}$.
\end{definition}

% \textbf{Tangent Space Structure:}
Let $\mathcal{Z} = U S V^\top \in \mathcal{M}_r$ satisfy the
weighted orthogonality conditions
\[
U^\top M_x U = V^\top M_y V = I_r,
\]
where $U \in \mathbb{R}^{m\times r}$,
$S \in \mathbb{R}^{r\times r}$, and
$V \in \mathbb{R}^{n\times r}$.
The tangent space of $\mathcal{M}_r$ at $\mathcal{Z}$ is %given by
\begin{equation*}
\label{eq:tangent_space}
T_{\mathcal{Z}}\mathcal{M}_r
=
\left\{
\delta U S V^\top
+
U\delta S V^\top
+
US\delta V^\top
:
U^\top M_x\delta U=0,\;
V^\top M_y\delta V=0
\right\},
\end{equation*}
where
$\delta U\in\mathbb{R}^{m\times r}$,
$\delta S\in\mathbb{R}^{r\times r}$, and
$\delta V\in\mathbb{R}^{n\times r}$.
Under the gauge conditions
\[
U^\top M_x\delta U=0,
\qquad
V^\top M_y\delta V=0,
\]
each $\delta Z\in T_{\mathcal{Z}}\mathcal{M}_r$ admits the unique
decomposition
\begin{align}
\delta Z
&=
\delta U S V^\top
+
U\delta S V^\top
+
US\delta V^\top
\notag\\
&=
P_U^\perp M_x\delta Z M_y P_V
+
P_U M_x\delta Z M_y P_V
+
P_U M_x\delta Z M_y P_V^\perp,
\label{eq:decomp_Z}
\end{align}
%The corresponding components are given by
where the corresponding components are given by
\begin{equation}\label{eq:decomp_USV}
\begin{aligned}
\delta U &= P_U^{\perp} M_x \delta Z M_y V S^{-1}, \\
\delta S &= U^{\top} M_x \delta Z M_y V, \\
\delta V &= P_V^{\perp} M_y \delta Z^{\top} M_x U S^{-\top}.
\end{aligned}
\end{equation}
The associated weighted projection operators are defined by
\begin{equation}
    \label{eq:projectors}
    P_U = U U^\top, \quad P_U^\perp = M_x^{-1} - P_U,
    \quad P_V = V V^\top, \quad P_V^\perp = M_y^{-1} - P_V.
\end{equation}

\begin{proposition}
Suppose that $\mathcal{Z}=USV^\top\in\mathcal{M}_r$ satisfies \Cref{def:weighted_dlra}, with initial data
\[
\mathcal{Z}(0)=U^0S^0(V^0)^\top\in\mathcal{M}_r,
\qquad
(U^0)^\top M_xU^0=I_r,
\qquad
(V^0)^\top M_yV^0=I_r.
\]
Then, WDLRA \eqref{eq:weighted_dlra_dr} is equivalent to the following formulations:
\begin{enumerate}
\item[(i)] \textbf{Galerkin formulation.}
Find $\partial_t\mathcal{Z}\in T_{\mathcal{Z}}\mathcal{M}_r$ such that
\begin{align*}
(\partial_t\mathcal{Z}-\mathscr{L}(\mathcal{Z}),\delta Z)_M=0,
\qquad
\forall\,\delta Z\in T_{\mathcal{Z}}\mathcal{M}_r,
\end{align*}
with the initial condition $\mathcal{Z}(0)=U^0S^0(V^0)^\top$.

\item[(ii)] \textbf{Equations of motion.}
The factors $U$, $S$, and $V$ satisfy
\begin{align*}
\dot U
&=P_U^\perp M_x\mathscr{L}(\mathcal{Z})M_yVS^{-1},
&
\dot S
&=U^\top M_x\mathscr{L}(\mathcal{Z})M_yV,
\\
\dot V
&=P_V^\perp M_y\mathscr{L}(\mathcal{Z})^\top M_xUS^{-\top},
\end{align*}
with initial conditions
$U(0) = U^0$, $S(0) = S^0$, $V(0) = V^0$.

\item[(iii)] \textbf{KLS system.}
Let $K=US$ and $L=VS^\top$. Then $K$, $L$, and $S$ satisfy
\begin{equation}\label{KLS}
    \begin{aligned}
        \dot K
&=\mathscr{L}(KV^\top)M_yV,
\quad 
\dot L
=\mathscr{L}(UL^\top)^\top M_xU,
\\
\dot S
&=U^\top M_x\mathscr{L}(USV^\top)M_yV,
    \end{aligned}
\end{equation}
with initial conditions
 $K(0) = U^0 S^0$, $L(0) = V^0 (S^0)^\top$, $S(0) = S^0$.
\end{enumerate}
\end{proposition}

\subsection{Implicit BUG integrator and BUG-MLFEM}\label{sec:bug}

To solve the WDLRA in \Cref{def:weighted_dlra}, we choose to solve the equivalent KLS system \eqref{KLS} using the basis-update \& Galerkin (BUG) integrator \cite{ceruti2022unconventional}. The backward Euler method is employed for the time discretization, resulting in an implicit BUG integrator. Given $\Delta t > 0$ and the factored rank-$r$ matrix
$\mathcal{W}^\mathfrak{n} = U^\mathfrak{n} S^\mathfrak{n} (V^\mathfrak{n})^\top$, with factors satisfying
\[
(U^\mathfrak{n})^\top M_x U^\mathfrak{n}
=
(V^\mathfrak{n})^\top M_y V^\mathfrak{n}
=
I_r,
\]
one time step yields the new rank-$r$ approximation at $t^{\mathfrak{n}+1}$,
\begin{equation}
    \mathcal{W}^{\mathfrak{n}+1}
    =
    U^{\mathfrak{n}+1} S^{\mathfrak{n}+1}
    (V^{\mathfrak{n}+1})^\top,
\end{equation}
where the updated factors satisfy
\[
(U^{\mathfrak{n}+1})^\top M_x U^{\mathfrak{n}+1}
=
(V^{\mathfrak{n}+1})^\top M_y V^{\mathfrak{n}+1}
=
I_r.
\]

We recall the generalized QR (GQR) decomposition introduced in
\cite{yin2025towards}. Given a symmetric positive definite matrix
$M\in\mathbb{R}^{m\times m}$, the GQR decomposition of
$K\in\mathbb{R}^{m\times r}$ is given by
\begin{equation}
    [U,R]=GQR(K,M),
\end{equation}
where the resulting factors $U \in \mathbb{R}^{m\times r}$ and $R\in \mathbb{R}^{r \times r}$ satisfy
\begin{equation}
    UR=K,
    \qquad
    U^\top M U=I_r.
\end{equation}
Then, the first-order implicit BUG
integrator is presented in \Cref{alg:KLS-Steps}.

\begin{remark}
The augmented BUG integrator \cite{ceruti2022rank} can also be incorporated into \Cref{alg:KLS-Steps}, and the algorithm is shown in \Cref{alg:AugBUG-Steps}. Since our primary objective is to develop efficient solvers, we focus on the standard BUG integrator for simplicity. The proposed solvers remain unchanged when applied to the augmented BUG setting.

\end{remark}

%=================================================================================
% KLS step Algorithm
%=================================================================================

\begin{algorithm}[htbp]
\caption{One step implicit BUG integrator for the KLS system \eqref{KLS}.}
\label{alg:KLS-Steps}
\begin{algorithmic}[1]

%\algcomment{\textbf{Input:}
\REQUIRE
$U^\mathfrak{n},\, S^\mathfrak{n},\, V^\mathfrak{n},\, \Delta t,\, M_x,\, M_y$.
\ENSURE
%\algcomment{\textbf{Output:} 
$U^{\mathfrak{n}+1},\, S^{\mathfrak{n}+1},\, V^{\mathfrak{n}+1}$.
\algcomment{\textbf{Step 1}: Update $U^\mathfrak{n} \to U^{\mathfrak{n}+1} \in \mathbb{R}^{m\times{r}}$ and $V^\mathfrak{n} \to V^{\mathfrak{n}+1} \in \mathbb{R}^{n\times{r}}$ in parallel:}
\STATE \textit{K-step:} Fix $V^\mathfrak{n}$, and solve $K^{\mathfrak{n}+1}$ from the $m \times r$ matrix equation
\begin{equation}
  D_t K^{\mathfrak{n}+1}
  = \mathscr{L}(K^{\mathfrak{n}+1} (V^\mathfrak{n})^\top)M_y V^\mathfrak{n},
  \qquad K^\mathfrak{n} = U^\mathfrak{n} S^\mathfrak{n}.
  \label{eq:k_eqn}
\end{equation}
\STATE Compute $[U^{\mathfrak{n}+1},\, R_K] = \mathrm{GQR}\!\left(K^{\mathfrak{n}+1},\, M_x\right)$.
\STATE Compute the $r \times r$ matrix $\mathcal{M} = (U^{\mathfrak{n}+1})^\top M_x U^\mathfrak{n}$.
\STATE \textit{L-step:} Fix $U^\mathfrak{n}$, and solve $L^{\mathfrak{n}+1}$ from the $n \times r$ matrix equation
\begin{equation}
  D_t L^{\mathfrak{n}+1}
  = \mathscr{L}(U^\mathfrak{n} (L^{\mathfrak{n}+1})^\top)^\top M_x U^\mathfrak{n},
  \qquad L^\mathfrak{n} = V^\mathfrak{n} (S^\mathfrak{n})^\top.
  \label{eq:l_eqn}
\end{equation}
\STATE Compute $[V^{\mathfrak{n}+1},\, R_L] = \mathrm{GQR}\!\left(L^{\mathfrak{n}+1},\, M_y\right)$.
\STATE Compute the $r \times r$ matrix $\mathcal{N} = (V^{\mathfrak{n}+1})^\top M_y V^\mathfrak{n}$.
\algcomment{\textbf{Step 2}: Update $S^\mathfrak{n} \to S^{\mathfrak{n}+1} \in \mathbb{R}^{{r}\times{r}}$:}
\STATE \textit{S-step:} Set $S^{\mathfrak{n},*} = \mathcal{M}\, S^\mathfrak{n}\, \mathcal{N}^\top \in \mathbb{R}^{{r}\times{r}}$.
\STATE Solve $S^{\mathfrak{n}+1}$ from the $r \times r$ matrix equation
\begin{equation}\label{eq:s_eqn}
  \frac{S^{\mathfrak{n}+1} - S^{\mathfrak{n},*}}{\Delta t}
  = (U^{\mathfrak{n}+1})^\top M_x\, \mathscr{L}\!\left(U^{\mathfrak{n}+1} S^{\mathfrak{n}+1} (V^{\mathfrak{n}+1})^\top\right) M_y V^{\mathfrak{n}+1}.
\end{equation}
\end{algorithmic}
\end{algorithm}

As an implicit algorithm, \Cref{alg:KLS-Steps} raises two fundamental
questions. The first is the well-posedness of the scheme, namely, whether
a unique solution exists at each time step. The second concerns
developing an efficient and low-cost solver for the resulting linear systems.

%=================================================================================
\subsection{Well-Posedness} 
%=================================================================================

To address the first question for \Cref{alg:KLS-Steps}, we introduce several subspaces of the finite element space \eqref{FEMspace}.
Given the low-rank state at time $t^\mathfrak{n}$ with the factorization
\(\mathcal W^\mathfrak{n}
= U^\mathfrak{n} S^\mathfrak{n} (V^\mathfrak{n})^\top\),
we define the following subspaces of $V_h$. These subspaces depend on the current low-rank factors and are dynamically updated at each time step:
\begin{subequations}
\begin{align}
V_0^{\mathfrak{n}}
  &= \left\{
  v \,\middle|\,
  v(x,y)
  = \Phi(x)^\top U^{\mathfrak{n}} S
  (V^{\mathfrak{n}})^\top \Psi(y),
  \quad \forall S \in \mathbb{R}^{r\times r}
  \right\},\\
  V_1^\mathfrak{n}
  &= \left\{
  v \,\middle|\,
  v(x,y)
  = \Phi(x)^\top K (V^\mathfrak{n})^\top \Psi(y),
  \quad \forall K \in \mathbb{R}^{m\times r}
  \right\},\\
  V_2^\mathfrak{n}
  &= \left\{
  v \,\middle|\,
  v(x,y)
  = \Phi(x)^\top U^\mathfrak{n} L^\top \Psi(y),
  \quad \forall L \in \mathbb{R}^{n\times r}
  \right\}.
\end{align}
\end{subequations}

It is readily observed that
\[
V_i^\mathfrak{n}\subset V_h,\qquad i=0,1,2,
\]
and
\[
\hat w_h^\mathfrak{n}
= \Phi(x)^\top U^\mathfrak{n} S^\mathfrak{n}
(V^\mathfrak{n})^\top \Psi(y)
\in V_0^\mathfrak{n}\cap V_1^\mathfrak{n}\cap V_2^\mathfrak{n}.
\]
In contrast, in general,
\[
\hat w_h^\mathfrak{n}\notin V_0^{\mathfrak{n}+1},
\]
since the latter subspace is determined by the updated factors
$U^{\mathfrak{n}+1}$ and $V^{\mathfrak{n}+1}$. Nevertheless, it holds that
\begin{equation}
    w_S^{\mathfrak{n},*}
    := \Phi(x)^\top U^{\mathfrak{n}+1} S^{\mathfrak{n},*}
    (V^{\mathfrak{n}+1})^\top \Psi(y)
    \in V_0^{\mathfrak{n}+1},
\end{equation}
where $S^{\mathfrak{n},*}$ is introduced in \Cref{alg:KLS-Steps}. Moreover, $w_S^{\mathfrak{n},*}$ is the mass-lumped $L^2$ projection of $\hat w_h^\mathfrak{n}$ onto $V_0^{\mathfrak{n}+1}$, characterized by
\begin{equation}\label{l2proj}
    (w_S^{\mathfrak{n},*},v)_h
    =(\hat w_h^\mathfrak{n},v)_h,
    \qquad \forall v\in V_0^{\mathfrak{n}+1}.
\end{equation}

The main advantage of WDLRA over DLRA is that each equation arising in the BUG integrator can be interpreted as an implicit MLFEM problem posed in a suitable subspace of the finite element space $V_h$.
The following result establishes the equivalence of \Cref{alg:KLS-Steps} and the first-order MLFEM formulation on the subspaces defined above.

\begin{lemma}\label{lem:mateqfem}
The matrices $K^{\mathfrak{n}+1}$, $L^{\mathfrak{n}+1}$, and
$S^{\mathfrak{n}+1}$ solve \eqref{eq:k_eqn}, \eqref{eq:l_eqn}, and
\eqref{eq:s_eqn} in \Cref{alg:KLS-Steps}, respectively, if and only if the 
functions
\begin{align*}
w_K^{\mathfrak{n}+1}
&:= \Phi(x)^\top K^{\mathfrak{n}+1}
(V^\mathfrak{n})^\top \Psi(y) \in V_1^\mathfrak{n},\\
w_L^{\mathfrak{n}+1}
&:= \Phi(x)^\top U^\mathfrak{n}
(L^{\mathfrak{n}+1})^\top \Psi(y) \in V_2^\mathfrak{n},\\
w_S^{\mathfrak{n}+1}
&:= \Phi(x)^\top U^{\mathfrak{n}+1} S^{\mathfrak{n}+1}
(V^{\mathfrak{n}+1})^\top \Psi(y) \in V_0^{\mathfrak{n}+1}
\end{align*}
solve the following first-order implicit BUG-MLFEM scheme:
\begin{align}
  (D_t w_K^{\mathfrak{n}+1}, w_1)_h
  + (\nabla w_K^{\mathfrak{n}+1}, \nabla w_1)_h
  - \alpha (w_K^{\mathfrak{n}+1}, w_1)_h &= 0,
  &&\forall w_1 \in V_1^{\mathfrak{n}}, \label{eq:Kvar32}\\
  (D_t w_L^{\mathfrak{n}+1}, w_2)_h
  + (\nabla w_L^{\mathfrak{n}+1}, \nabla w_2)_h
  - \alpha (w_L^{\mathfrak{n}+1}, w_2)_h &= 0,
  &&\forall w_2 \in V_2^{\mathfrak{n}}, \label{eq:Lvar32}\\
  (D_t w_S^{\mathfrak{n}+1}, w_0)_h
  + (\nabla w_S^{\mathfrak{n}+1}, \nabla w_0)_h
  - \alpha (w_S^{\mathfrak{n}+1}, w_0)_h &= 0,
  &&\forall w_0 \in V_0^{\mathfrak{n}+1}, \label{eq:Svar32}
\end{align}
where $w_K^\mathfrak{n} = w_L^\mathfrak{n} = w_S^\mathfrak{n} = \hat w_h^\mathfrak{n}$.
\end{lemma}

The proof of \Cref{lem:mateqfem} follows from the matrix representations of the inner products established in \Cref{lem:inner-products}, together with \eqref{l2proj}.

\begin{definition}\label{def:lows}
At time $t^{\mathfrak{n}+1}$, we define $V_0^{\mathfrak{n}+1}$ as the BUG-MLFEM space and $\hat w_h^{\mathfrak{n}+1}:=w_S^{\mathfrak{n}+1} \in V_0^{\mathfrak{n}+1}$ as the BUG-MLFEM solution.
\end{definition}

\begin{theorem}[Well-posedness]
\label{thm:wp32}
Given the BUG-MLFEM solution $\hat w_h^{\mathfrak{n}}$, with
\(
w_K^{\mathfrak{n}}
=
w_L^{\mathfrak{n}}
=
w_S^{\mathfrak{n}}
=
\hat w_h^{\mathfrak{n}}.
\)
Then, (i) for any $\Delta t>0$, the first-order implicit BUG-MLFEM scheme
\eqref{eq:Kvar32}--\eqref{eq:Svar32} admits a unique solution
\begin{equation}\label{wkls}
    (w_K^{\mathfrak{n}+1},\,w_L^{\mathfrak{n}+1},\,w_S^{\mathfrak{n}+1})
\in
V_1^{\mathfrak{n}}
\times
V_2^{\mathfrak{n}}
\times
V_0^{\mathfrak{n}+1}.
\end{equation}
Equivalently, the matrix system in \Cref{alg:KLS-Steps} admits a unique
solution
\begin{equation}
(K^{\mathfrak{n}+1},\,L^{\mathfrak{n}+1},\,S^{\mathfrak{n}+1}) \in \mathbb{R}^{m\times r}\times \mathbb{R}^{n\times r}\times \mathbb{R}^{r\times r}.
\end{equation}
(ii) The BUG-MLFEM solution satisfies the mass-lumped
$L^2$-stability
\begin{equation}\label{l2stab}
\|\hat w_h^{\mathfrak{n}+1}\|_h
\leq
\|\hat w_h^{\mathfrak{n}}\|_h. 
\end{equation}
\end{theorem}
\begin{proof}
(i) We present the proof for the $S$-step in detail; the $K$- and $L$-steps follow by the same reasoning. Since \eqref{eq:Svar32} is a finite-dimensional linear system, existence is equivalent to uniqueness. Thus, it suffices to prove uniqueness.

Suppose that $w_{S,1}^{\mathfrak{n}+1}$ and $w_{S,2}^{\mathfrak{n}+1}$ are two solutions of \eqref{eq:Svar32} corresponding to the same data $\hat w_h^\mathfrak{n}$. Define
\(
\delta w_S^{\mathfrak{n}+1}
:=w_{S,1}^{\mathfrak{n}+1}-w_{S,2}^{\mathfrak{n}+1}
\in V_0^{\mathfrak{n}+1}.
\)
Subtracting the two equations gives
\[
\bigl(D_t\delta w_S^{\mathfrak{n}+1},w_0\bigr)_h
+A(\delta w_S^{\mathfrak{n}+1},w_0)
=\alpha\bigl(\delta w_S^{\mathfrak{n}+1},w_0\bigr)_h,
\quad
\forall\,w_0\in V_0^{\mathfrak{n}+1}.
\]
Since both solutions share the same previous-time data, 
\(
\delta w_S^{\mathfrak{n}}=0,
\)
and hence
\[
D_t\delta w_S^{\mathfrak{n}+1}
=\frac{\delta w_S^{\mathfrak{n}+1}}{\Delta t}.
\]
Taking $w_0=\delta w_S^{\mathfrak{n}+1}$ yields
\[
\frac{1}{\Delta t}
\|\delta w_S^{\mathfrak{n}+1}\|_h^2
+\|\nabla\delta w_S^{\mathfrak{n}+1}\|_{L^2(\Omega)}^2
=
\alpha\|\delta w_S^{\mathfrak{n}+1}\|_h^2,
\]
where we have used
\(
A(v,v)=\|\nabla v\|_{L^2(\Omega)}^2.
\)
Therefore,
\[
\left(\frac{1}{\Delta t}-\alpha\right)
\|\delta w_S^{\mathfrak{n}+1}\|_h^2
+\|\nabla\delta w_S^{\mathfrak{n}+1}\|_{L^2(\Omega)}^2
=0.
\]
Since $\Delta t>0$ and $\alpha\le0$, we have
\(
\frac{1}{\Delta t}-\alpha>0.
\)
Both terms in the above identity are therefore nonnegative, which implies
\(
\|\delta w_S^{\mathfrak{n}+1}\|_h=0.
\)
Hence,
\(
\delta w_S^{\mathfrak{n}+1}=0,
\)
and consequently
\[
w_{S,1}^{\mathfrak{n}+1}=w_{S,2}^{\mathfrak{n}+1}.
\]
Thus, the $S$-step solution is unique.

The existence and uniqueness of the solution
$(K^{\mathfrak{n}+1},L^{\mathfrak{n}+1},S^{\mathfrak{n}+1})$
in \Cref{alg:KLS-Steps} then follow from those of \eqref{wkls} and \Cref{lem:mateqfem}.

(ii) Choose $w_0 = w_S^{\mathfrak{n}+1}$ in \eqref{eq:Svar32} to obtain
\[
  \bigl(D_t w_S^{\mathfrak{n}+1}, w_S^{\mathfrak{n}+1}\bigr)_h
  + \|\nabla w_S^{\mathfrak{n}+1}\|_{L^2(\Omega)}^2
  = \alpha\|w_S^{\mathfrak{n}+1}\|_h^2.
\]
Applying the polarization identity
\(
  (a-b,a)
  = \frac{1}{2}
  \left(
    \|a\|^2-\|b\|^2+\|a-b\|^2
  \right)
\)
with $a=w_S^{\mathfrak{n}+1}$ and $b=w_S^\mathfrak{n}$, and multiplying by $2\Delta t$, gives
\[
  \|w_S^{\mathfrak{n}+1}\|_h^2
  -\|w_S^\mathfrak{n}\|_h^2
  +\|w_S^{\mathfrak{n}+1}-w_S^\mathfrak{n}\|_h^2
  +2\Delta t
  \|\nabla w_S^{\mathfrak{n}+1}\|_{L^2(\Omega)}^2
  =
  2\Delta t\,\alpha
  \|w_S^{\mathfrak{n}+1}\|_h^2.
\]
Since $\Delta t>0$ and $\alpha\leq0$, the right-hand side is nonpositive, while the third and fourth terms on the left-hand side are nonnegative. Hence,
\[
  \|w_S^{\mathfrak{n}+1}\|_h^2
  \leq
  \|w_S^\mathfrak{n}\|_h^2.
\]
By \Cref{def:lows}, this is equivalent to \eqref{l2stab}.
\end{proof}

\section{Solvers for the implicit BUG integator}\label{sec:solvers}

To address the second question in \Cref{sec:bug} concerning \Cref{alg:KLS-Steps}, we develop low-cost non-iterative Sylvester solvers for the linear systems arising in the $K$- and $L$-steps.

\subsection{Structure of KLS linear systems}
\label{sec:linear_systems}

We begin by examining the structure of the linear systems involved.
%The structure will be exploited to develop efficient solvers for the matrix equations associated with the $K$- and $L$-steps.
Substituting the operator $\mathscr{L}$ from
\eqref{eq:LxLy} into \eqref{eq:k_eqn} gives
\begin{equation*}
\frac{K^{\mathfrak{n}+1}-K^\mathfrak{n}}{\Delta t}
= L_xK^{\mathfrak{n}+1}
- K^{\mathfrak{n}+1}(V^\mathfrak{n})^\top A_yV^\mathfrak{n}
+ \alpha K^{\mathfrak{n}+1}.
\end{equation*}
Rearranging the terms yields the $K$-equation
\begin{equation}\label{Ksylvester}
(-\Delta t\, L_x)K^{\mathfrak{n}+1}
+K^{\mathfrak{n}+1}
\Bigl((1-\Delta t\,\alpha)I
+\Delta t\,(V^\mathfrak{n})^\top A_yV^\mathfrak{n}\Bigr)
=K^\mathfrak{n}.
\end{equation}
Similarly, substituting \eqref{eq:LxLy} into \eqref{eq:l_eqn} gives the $L$-equation
\begin{equation}\label{Lsylvester}
(-\Delta t\, L_y)L^{\mathfrak{n}+1}
+L^{\mathfrak{n}+1}
\Bigl((1-\Delta t\,\alpha)I
+\Delta t\,(U^\mathfrak{n})^\top A_xU^\mathfrak{n}\Bigr)
=L^\mathfrak{n}.
\end{equation}
Finally, \eqref{eq:s_eqn} gives the $S$-equation
\begin{equation}\label{Ssylvester}
\resizebox{\linewidth}{!}{$\displaystyle
\Bigl((1-\Delta t\,\alpha) I
+\Delta t\,(U^{\mathfrak{n}+1})^\top A_xU^{\mathfrak{n}+1}\Bigr)
S^{\mathfrak{n}+1}
+
S^{\mathfrak{n}+1}
\Bigl(\Delta t\,(V^{\mathfrak{n}+1})^\top A_yV^{\mathfrak{n}+1}\Bigr)
=
S^{\mathfrak{n},*}.$
%\mathcal{M}S^{\mathfrak{n},*}\mathcal{N}^\top.$
}
\end{equation}

Thus, all three linear matrix equations
\eqref{Ksylvester}--\eqref{Ssylvester}
can be expressed in the unified Sylvester form
\begin{equation}\label{eq:nd_sylvester}
AX+XB=C.
\end{equation}
More precisely, the corresponding coefficient matrices, unknowns, and
right-hand sides are given as follows.

\begin{itemize}
    \item For the $K$-equation \eqref{Ksylvester},
    \[
    A=-\Delta t\,L_x\in\mathbb{R}^{m\times m},\qquad
    B=(1-\Delta t\,\alpha)I_r
    +\Delta t\,(V^\mathfrak{n})^\top A_yV^\mathfrak{n}
    \in\mathbb{R}^{r\times r},
    \]
    \[
    X=K^{\mathfrak{n}+1}\in\mathbb{R}^{m\times r},\qquad
    C=K^\mathfrak{n}\in\mathbb{R}^{m\times r}.
    \]

    \item For the $L$-equation \eqref{Lsylvester},
    \[
    A=-\Delta t\,L_y\in\mathbb{R}^{n\times n},\qquad
    B=(1-\Delta t\,\alpha)I_r
    +\Delta t\,(U^\mathfrak{n})^\top A_xU^\mathfrak{n}
    \in\mathbb{R}^{r\times r},
    \]
    \[
    X=L^{\mathfrak{n}+1}\in\mathbb{R}^{n\times r},\qquad
    C=L^\mathfrak{n}\in\mathbb{R}^{n\times r}.
    \]

    \item For the $S$-equation \eqref{Ssylvester},
    \[
    A=(1-\Delta t\,\alpha)I
    +\Delta t\,(U^{\mathfrak{n}+1})^\top A_xU^{\mathfrak{n}+1}
    \in\mathbb{R}^{r\times r},
    \]
    \[
    B=\Delta t\,(V^{\mathfrak{n}+1})^\top A_yV^{\mathfrak{n}+1}
    \in\mathbb{R}^{r\times r},
    \]
    \[
    X=S^{\mathfrak{n}+1}\in\mathbb{R}^{r\times r},\qquad
    C=S^{\mathfrak{n},*}
    %C=\mathcal{M}S^{\mathfrak{n},*}\mathcal{N}^\top
    \in\mathbb{R}^{r\times r}.
    \]
\end{itemize}

\subsection{Conventional Sylvester solver}
A direct application of a conventional dense Sylvester equation solver,
such as the Bartels--Stewart algorithm \cite{bartels1972algorithm}, would require Schur
decompositions of the coefficient matrices. For the $K$- and
$L$-equations, the resulting computational costs are
$O(m^3)$ and $O(n^3)$, respectively, while the $S$-equation incurs
$O(r^3)$ operations. Consequently, solving all three equations by a
standard dense Sylvester solver would require
\(
O(m^3+n^3+r^3)
\)
operations. Since the rank $r$ is typically much smaller than the
spatial dimensions $m$ and $n$, the dominant cost is
$O(m^3+n^3)$. This is comparable to the computational complexity of
the corresponding full-rank discretization and therefore does not
fully exploit the low-rank structure.

For the $S$-equation, all coefficient matrices are of size
$r\times r$, so a conventional Sylvester solver has a cost of only
$O(r^3)$ and is inexpensive when $r\ll m,n$. In contrast, the
Sylvester equations \eqref{eq:nd_sylvester} arising from the $K$- and
$L$-steps possess additional structure that can be exploited to
further reduce the computational cost. In particular, the spatial
coefficient matrices in the $K$- and $L$-equations are
$A=-\Delta t\,L_x$ and $A=-\Delta t\,L_y$, respectively, and thus
inherit the sparsity of the underlying one-dimensional spatial
discretization operators. More specifically, since $M_x$ and $M_y$
are diagonal matrices, the matrices $L_x$ and $L_y$ defined in
\eqref{eq:1d-diffusion-operators} have the same sparsity patterns as
$A_x$ and $A_y$, respectively. Meanwhile, the corresponding
coefficient matrices $B$ are only of size $r\times r$ and are
generally symmetric. Hence, the $K$- and $L$-equations exhibit a
sparse-by-small structure that is not exploited by conventional
dense Sylvester solvers, motivating the development of specialized
solvers that take advantage of this structure.

The storage cost follows the same pattern. The Bartels--Stewart algorithm \cite{bartels1972algorithm} requires the real Schur factors of $A$, which are generally dense $m\times m$ matrices for the $K$-equation (resp.\ $n\times n$ matrices for the $L$-equation). Thus, although $A$ itself is banded, its Schur factors generally do not preserve this sparsity, resulting in $O(m^2)$ and $O(n^2)$ storage, respectively. For the $S$-equation, all coefficient matrices are of size $r\times r$, requiring only $O(r^2)$ storage. Hence, the total storage cost is $O(m^2+n^2+r^2)=O(m^2+n^2)$, where the latter follows from $r\ll m,n$. In particular, the dense Schur factors lead to a storage requirement comparable to that of the corresponding full-rank discretization.

\subsection{Low-cost non-iterative Sylvester solvers}
Both $K$-equation \eqref{Ksylvester} and $L$-equation \eqref{Lsylvester} have the
unified Sylvester form \eqref{eq:nd_sylvester}, with the spatial
coefficient matrix $A$ sharing the same sparsity pattern as $A_x$ and
$A_y$ in the $K$- and $L$-equations, respectively, while the
corresponding low-dimensional coefficient matrix $B$ is symmetric.
We therefore develop a solver for \eqref{eq:nd_sylvester} that
exploits these sparsity and symmetry properties. We focus on the
$K$-equation \eqref{Ksylvester}, with the application for the $L$-equation
following analogously.

\subsubsection{Banded stiffness matrix: Dirichlet and Neumann BCs}

Under the Neumann and Dirichlet BCs, both $A_x$ and $A_y$ in \eqref{AxAy}
are symmetric banded matrices with bandwidth $2k+1$ and half-bandwidth $k$; namely,
\[
(A_x)_{ij}=(A_y)_{ij}=0
\qquad \text{for all } |i-j|>k.
\]
In the Sylvester equation \eqref{eq:nd_sylvester}, the matrix $A = -\Delta t L_x =  \Delta t M_x^{-1} A_x\in\mathbb{R}^{m\times m}$ or $A = -\Delta t L_y =  \Delta t M_y^{-1} A_y\in\mathbb{R}^{n\times n}$ inherits the same bandwidth $2k+1$ and half-bandwidth $k$. However, since
$M_x^{-1}$ and $M_y^{-1}$ are generally not scalar multiples of the
identity, $A$ is generally nonsymmetric, despite the symmetry of $A_x$ and
$A_y$. In contrast, the matrix $B\in\mathbb{R}^{r\times r}$ is symmetric.

The key idea is therefore to diagonalize the small matrix $B$ rather
than the large matrix $A$. Eigendecomposing $B\in\mathbb{R}^{r\times r}$
costs $O(r^3)$, whereas diagonalizing
$A\in\mathbb{R}^{m\times m}$ would cost $O(m^3)$ and, more
importantly, destroy its banded structure, since the eigenvectors of a
banded matrix are generally dense. Because $r\ll m$, diagonalizing $B$
is therefore the natural choice: it preserves the sparsity of the
large matrix $A$ while restricting the eigendecomposition to the
small matrix $B$. 

The steps for solving the unified Sylvester equation \eqref{eq:nd_sylvester} are as follows.

\subsection*{Step 1: Eigendecomposition of $B$}

Since $B$ is symmetric, it admits a real orthogonal eigendecomposition
\begin{equation}
    B = Q\Lambda Q^\top,
    \qquad
    \Lambda = \operatorname{diag}(\lambda_1,\dots,\lambda_r),
    \qquad
    Q^\top Q = I_r,
    \qquad
    BQ = Q\Lambda.
    \label{eq:eig-B}
\end{equation}
Here, the columns $q_1,\dots,q_r$ of $Q$ form an orthonormal basis
of eigenvectors of $B$, satisfying
$Bq_j=\lambda_jq_j$, where the eigenvalues
$\lambda_1,\dots,\lambda_r$ are real.

\subsection*{Step 2: Diagonalization of the Sylvester equation}

Right-multiplying \eqref{eq:nd_sylvester} by $Q$ and using
$BQ=Q\Lambda$ gives
\begin{equation}
    AXQ + XBQ = CQ
    \;\Longrightarrow\;
    A(XQ) + (XQ)\Lambda = CQ.
    \label{eq:right-mult}
\end{equation}
Introducing
\[
    Y := XQ,
    \qquad
    \widetilde{C} := CQ,
\]
the unified Sylvester equation \eqref{eq:nd_sylvester} can be equivalently written as
%we obtain the equivalent equation
\begin{equation}
    AY + Y\Lambda = \widetilde{C},
    \label{eq:decoupled-matrix}
\end{equation}
in which the dense symmetric matrix $B$ is replaced by the
diagonal matrix $\Lambda$.

\subsection*{Step 3: Column-wise solution of the new Sylvester equation}

Let $y_j,\widetilde{c}_j\in\mathbb{R}^m$ denote the $j$-th columns
of $Y$ and $\widetilde{C}$, respectively. Since $\Lambda$ is diagonal,
\eqref{eq:decoupled-matrix} decouples into $r$ independent shifted
linear systems:
\begin{equation}
    (A+\lambda_j I)y_j=\widetilde{c}_j,
    \qquad
    j=1,\dots,r.
    \label{eq:shifted-system}
\end{equation}
Each matrix $A+\lambda_j I$ differs from $A$ only in its diagonal
entries and therefore retains the half-bandwidth $k$. 
The well-posedness of \eqref{eq:shifted-system} is guaranteed by \Cref{thm:wp32}.
Consequently, each system in
\eqref{eq:shifted-system} can be solved independently by banded
Gaussian elimination.

\subsection*{Step 4: Back transformation}

After solving \eqref{eq:shifted-system}, assemble
\[
    Y=[y_1,\dots,y_r]\in\mathbb{R}^{m\times r}.
\]
Since $Y=XQ$ and $Q$ is orthogonal, the solution is recovered by
\begin{equation}
    X=YQ^\top.
    \label{eq:back-transform}
\end{equation}

\Cref{alg:sylvester-sym} summarizes the above procedure.
It replaces a dense Sylvester solve whenever an equation of the form
$AX+XB=C$ arises with banded $A$ and symmetric $B$. In particular,
it is applied to the K- and L-step equations
\eqref{eq:k_eqn} and \eqref{eq:l_eqn} of the BUG integrator after
they are cast into the unified form \eqref{eq:nd_sylvester}.

\begin{algorithm}[htbp]
\caption{Low-cost solver for \eqref{eq:nd_sylvester} with $A$ banded and $B$ symmetric.}
\label{alg:sylvester-sym}
\begin{algorithmic}[1]
\REQUIRE
$A \in \mathbb{R}^{m\times m}$ banded with half-bandwidth $k$,
$B=B^\top \in \mathbb{R}^{r\times r}$,
$C\in\mathbb{R}^{m\times r}$.
\ENSURE
$X\in\mathbb{R}^{m\times r}$ satisfying $AX+XB=C$.

\algcomment{\textbf{Step 1}: Eigendecomposition of $B$:}
\STATE Compute the real orthogonal eigendecomposition
$B=Q\Lambda Q^\top$, where
$\Lambda=\operatorname{diag}(\lambda_1,\ldots,\lambda_r)$
and $Q^\top Q=I_r$.

\algcomment{\textbf{Step 2}: Diagonalization of the Sylvester equation:}
\STATE Set
$\widetilde{C} = CQ$ and $Y=XQ$, so that
$AY+Y\Lambda=\widetilde{C}$.

\algcomment{\textbf{Step 3}: Column-wise solution of the new Sylvester equation:}
\FOR{$j=1,\ldots,r$}
    \STATE Solve
    $(A+\lambda_j I)y_j=\widetilde{c}_j$
    by banded Gaussian elimination.
\ENDFOR

\STATE Assemble
$Y=[y_1,\dots,y_r]$.

\algcomment{\textbf{Step 4}: Back transformation:}
\STATE Recover the solution by
$X= YQ^\top$.

\end{algorithmic}
\end{algorithm}

\begin{remark}
\Cref{alg:sylvester-sym} can also be applied to mixed boundary conditions involving homogeneous Dirichlet and Neumann conditions.
\end{remark}

The computational complexity of \Cref{alg:sylvester-sym} is mainly determined by the four steps. Since $B\in\mathbb{R}^{r\times r}$ is symmetric, its eigendecomposition costs $O(r^3)$. The two matrix multiplications $\widetilde{C}=CQ$ and $X=YQ^\top$ each require $O(mr^2)$ operations. In Step~3, each shifted matrix $A+\lambda_j I$ has the same half-bandwidth $k$ as $A$ and can therefore be factorized using banded Gaussian elimination in $O(mk^2)$ operations, while the corresponding forward and backward substitutions require $O(mk)$ operations. Since the $r$ shifted systems are independent, their total cost is $O(rmk^2)$, with $O(rmk)$ additional work for the triangular solves. The back transformation $X=YQ^\top$ in Step~4 requires
$O(mr^2)$ operations. Consequently, the overall computational complexity of \Cref{alg:sylvester-sym} is $O\!\left(r^3+mr^2+mrk^2\right)$,
where the lower-order $O(rmk)$ term has been omitted. In particular, since $k$ is typically fixed and
$r\ll m$, the overall computational complexity is dominated by
$O(mr^2)$.

Regarding the storage cost, the banded matrix $A$ requires
$O(mk)$ storage, while the symmetric matrix $B$ and its
eigenvector matrix $Q$ require $O(r^2)$ storage. The matrices
$C$, $Y$, and $X$ each require $O(mr)$ storage. Since the
$r$ shifted banded systems in Step~3 can be solved sequentially, only
one banded LU factorization needs to be stored at a time, requiring
$O(mk)$ additional storage. Hence, the overall storage
requirement is $O(mk+mr+r^2)$.
In particular, since $k$ is typically fixed and $r\ll m$, the storage
cost is essentially $O(mr)$.

\subsubsection{Banded stiffness matrix with wrap-around entries: periodic BCs}

Since the basis functions associated with the nodes at the two ends of the domain are identified under periodic BCs, the periodicity introduces additional couplings between the first and last degrees of freedom. Consequently, the matrices $A_x$ and $A_y$ in \eqref{AxAy} retain the same $2k+1$ banded structure as in the Neumann and Dirichlet cases, but acquire additional nonzero entries outside the band. Nevertheless, both $A_x$ and $A_y$ remain symmetric.

\begin{lemma}\label{lem:psd}
Let $A_x$ and $A_y$ denote the stiffness matrices associated with the
one-dimensional periodic finite element discretizations in the
$x$- and $y$-directions, respectively. Then, the following results hold.
\begin{enumerate}
    \item[(i)] $A_x$ and $A_y$ are positive semidefinite, namely,
    \[
        A_x\succeq0,
        \qquad
        A_y\succeq0.
    \]
    
    \item[(ii)]
    Let $\widehat{A}_x$ and $\widehat{A}_y$ denote the
    $(2k+1)$-banded parts of $A_x$ and $A_y$, respectively, obtained by
    retaining only the entries within the $(2k+1)$ banded diagonals.
    Then, $\widehat{A}_x$ and $\widehat{A}_y$ are also positive semidefinite, namely,
    \[
        \widehat{A}_x\succeq0,
        \qquad
        \widehat{A}_y\succeq0.
    \]
\end{enumerate}
\end{lemma}

\begin{proof}
It suffices to prove the assertions for $A_x$ and $\widehat{A}_x$;
the arguments for $A_y$ and $\widehat{A}_y$ are analogous.

{(i)}
Let $v=(v_1,\ldots,v_m)^\top\in\mathbb{R}^m$ be arbitrary, and let
$v_h$ be the corresponding finite element function,
\(
    v_h(x)=\sum_{j=1}^m v_j\phi_j(x),
\)
where $\{\phi_j\}_{j=1}^m$ denotes the finite element basis.
By the definition of the stiffness matrix,
\begin{equation}
    v^\top A_x v
    =
    \sum_{i=1}^{M}
    \int_{I_i^x}
    \left(v_h'(x)\right)^2\,dx
    \ge 0.
\end{equation}
Hence,
\[
    A_x\succeq0.
\]

{(ii)}
For each element $I_i^x$, let
\[
    K^{(i)}
    =
    \left[
        \int_{I_i^x}
        \phi_{(i-1)k+p}'(x)
        \phi_{(i-1)k+q}'(x)\,dx
    \right]_{p,q=1}^{k+1}
    \in\mathbb{R}^{(k+1)\times(k+1)}
\]
denote the corresponding local stiffness matrix. For any
$z=(z_1,\ldots,z_{k+1})^\top\in\mathbb{R}^{k+1}$,
\[
\begin{aligned}
    z^\top K^{(i)}z
    &=
    \int_{I_i^x}
    \left(
        \sum_{p=1}^{k+1}
        z_p\phi_{(i-1)k+p}'(x)
    \right)^2\,dx
    \ge0.
\end{aligned}
\]
Thus,
\[
    K^{(i)}\succeq0,
    \qquad i=1,\ldots,M.
\]

For $i=1,\ldots,M-1$, the standard finite element assembly gives
\[
    R_i^\top K^{(i)}R_i,
\]
where $R_i \in \mathbb{R}^{(k+1)\times m}$ denotes the local-to-global restriction matrix.

For the last element $I_M^x$, the periodic identification of the
endpoints causes the last local degree of freedom to coincide with
the first global degree of freedom. Consequently, the corresponding
entries coupling these two sets of degrees of freedom appear in the
corner blocks of $A_x$ and lie outside the $(2k+1)$ banded diagonals.
Removing these entries amounts to replacing $K^{(M)}$ by
\[
    \widehat{K}^{(M)}
    =
    \begin{pmatrix}
        K^{(M)}_{1:k,1:k} & 0\\
        0 & K^{(M)}_{k+1,k+1}
    \end{pmatrix}.
\]
Since $K^{(M)}\succeq0$, its principal submatrix satisfies
\[
    K^{(M)}_{1:k,1:k}\succeq0.
\]
Moreover,
\[
    K^{(M)}_{k+1,k+1}
    =
    \int_{I_M^x}
    \left|\phi_m'(x)\right|^2\,dx
    \ge0.
\]
Therefore,
\[
    \widehat{K}^{(M)}\succeq0.
\]

It follows that the $(2k+1)$-banded part of $A_x$ admits the
decomposition
\[
    \widehat{A}_x
    =
    \sum_{i=1}^{M-1}
    R_i^\top K^{(i)}R_i
    +
    R_M^\top\widehat{K}^{(M)}R_M,
\]
where $R_M$ denotes the local-to-global restriction matrix without
the periodic wrap-around contribution.

Therefore, for any $v\in\mathbb{R}^m$,
\begin{align}
    v^\top\widehat{A}_xv
    &=
    \sum_{i=1}^{M-1}
    (R_iv)^\top K^{(i)}(R_iv)
    +
    (R_Mv)^\top
    \widehat{K}^{(M)}
    (R_Mv)
    \ge0.
\end{align}
Hence,
\[
    \widehat{A}_x\succeq0.
\]
The proof is complete.
\end{proof}

The matrix $A$ in \eqref{eq:nd_sylvester} inherits the sparsity pattern of $A_x$ and $A_y$ for the $K$- and $L$-equations \eqref{Ksylvester} and \eqref{Lsylvester}, respectively. However, unlike $A_x$ and $A_y$, the matrix $A$ is generally nonsymmetric. In particular, $A$ has a banded structure with bandwidth $2k+1$ and half-bandwidth $k$, together with $2k$ additional nonzero entries induced by the periodic coupling.
More specifically, these $2k$ additional nonzero entries occur at
\[
A(1,m-\ell+1)\neq 0,\qquad
A(m-\ell+1,1)\neq 0,\qquad
\ell=1,\ldots,k,
\]
which are the last $k$ entries in the first row and first column, respectively.
Thus, $A$ can be viewed as a matrix with half-bandwidth $k$ and additional nonzero entries connecting the two ends of the domain.
As in the Neumann and Dirichlet cases, the matrix $B\in\mathbb{R}^{r\times r}$ in \eqref{eq:nd_sylvester} remains symmetric.

Because of the additional nonzero entries outside the $2k+1$ banded diagonals of $A$, the solver for \eqref{eq:nd_sylvester} developed in \Cref{alg:sylvester-sym} is no longer directly applicable. Therefore, additional treatment is required to account for these extra terms. To take advantage of \Cref{alg:sylvester-sym}, we decompose $A$ as
\begin{equation}
    \label{eq:A-periodic-structure}
    A = \widehat{A} + \widehat U \widehat{V}^\top,
\end{equation}
where $\widehat{A}\in\R^{m\times m}$ denotes the $2k+1$-banded part of $A$, and $\widehat U, \widehat V\in\R^{m\times 2}$ are the rank-$2$ Woodbury factors given by
\begin{equation}
    \label{eq:woodbury-factors}
    \widehat U =
    \begin{pmatrix}
        e_1 & v
    \end{pmatrix},
    \qquad
    \widehat V =
    \begin{pmatrix}
        u & e_1
    \end{pmatrix}.
\end{equation}
Here, $e_1\in\R^m$ is the first canonical basis vector, i.e.,
\begin{equation}
    e_1(j)=\delta_{1j},
\end{equation}
where $\delta_{ij}$ denotes the Kronecker delta. The vectors $u,v\in\R^m$ are defined by
\begin{equation}
    \label{eq:u-vector}
    u =
    \begin{pmatrix}
        \mathbf{0}_{m-k}\\
        A_{1,m-k+1}\\
        A_{1,m-k+2}\\
        \vdots\\
        A_{1,m}
    \end{pmatrix},
    \qquad
    v =
    \begin{pmatrix}
        \mathbf{0}_{m-k}\\
        A_{m-k+1,1}\\
        A_{m-k+2,1}\\
        \vdots\\
        A_{m,1}
    \end{pmatrix}.
\end{equation}
Consequently,
\begin{equation}
    \widehat U\widehat{V}^\top
    =
    \begin{pmatrix}
        e_1 & v
    \end{pmatrix}
    \begin{pmatrix}
        u^\top\\
        e_1^\top
    \end{pmatrix}
    =
    e_1u^\top + ve_1^\top,
\end{equation}
so that the rank-$2$ correction accounts for the additional nonzero entries of $A$ outside its $2k+1$ banded diagonals.

To solve \eqref{eq:nd_sylvester} with $A$ of the structure
\eqref{eq:A-periodic-structure}, we only need to modify
\textbf{Step 3} of the algorithm proposed for the Neumann and
Dirichlet boundary cases. This can be achieved using the
Sherman--Morrison--Woodbury identity \cite{woodbury1950inverting},
whereby each shifted linear system is reduced to a purely banded solve
plus a correction.

For the algorithm with periodic BCs, we first repeat
\textbf{Steps 1 and 2} of \Cref{alg:sylvester-sym} for the Neumann and
Dirichlet boundary cases. The modified \textbf{Step 3} for the periodic
boundary case is then given as follows.

\subsection*{Step 3: Woodbury-corrected banded solve}

Let $y_j,\widetilde{c}_j\in\mathbb{R}^m$ denote the $j$-th columns of
$Y$ and $\widetilde{C}$, respectively. Since $\Lambda$ obtained in \textbf{Step 1} is diagonal,
\eqref{eq:decoupled-matrix} decouples into $r$ independent shifted
linear systems:
\begin{equation}
    (A+\lambda_j I)y_j=\widetilde{c}_j,
    \qquad
    j=1,\dots,r.
    \label{eq:shifted-systemperiod}
\end{equation}

Using the decomposition \eqref{eq:A-periodic-structure}, we have
\begin{equation}
    A+\lambda_j I
    =
    \widehat{A}_j+\widehat{U}\widehat{V}^{\top},
    \qquad
    \widehat{A}_j:=\widehat{A}+\lambda_j I.
    \label{eq:shifted-periodic-decomp}
\end{equation}
Since the shift $\lambda_j I$ only modifies the diagonal entries,
$\widehat{A}_j$ remains a purely banded matrix with half-bandwidth $k$.

\begin{lemma}\label{Ajnons}
For each $j=1,\ldots,r$,
\begin{enumerate}
    \item[(i)] the matrices $\widehat{A}_j$ and
    $\widehat{A}_j+\widehat{U}\widehat{V}^{\top}$ in
    \eqref{eq:shifted-periodic-decomp} are nonsingular;
    \item[(ii)] the matrix
    \(
    I_2+\widehat{V}^{\top}\widehat{A}_j^{-1}\widehat{U} \in \mathbb{R}^{2\times2}
\)
    is nonsingular.
\end{enumerate}
\end{lemma}

\begin{proof}
(i) Recall that
\[
    A=-\Delta t\,L_x
    =\Delta t\,M_x^{-1}A_x
    \qquad\text{or}\qquad
    A=-\Delta t\,L_y
    =\Delta t\,M_y^{-1}A_y.
\]
Let $\widehat{A}_x$ and $\widehat{A}_y$ be the matrices defined in
\Cref{lem:psd}. Then, for the matrix $\widehat{A}$ defined in
\eqref{eq:A-periodic-structure}, we have
\[
    \widehat{A}
    =\Delta t\,M_x^{-1}\widehat{A}_x
    \qquad\text{or}\qquad
    \widehat{A}
    =\Delta t\,M_y^{-1}\widehat{A}_y.
\]
Since $M_x\succ0$ and $\widehat{A}_x\succeq0$ by
\Cref{lem:psd},
\[
    M_x^{-1}\widehat{A}_x
    \sim
    M_x^{-1/2}\widehat{A}_xM_x^{-1/2}
    \succeq0.
\]
Hence, all eigenvalues of $\widehat{A}$ are real and nonnegative.
The same conclusion follows when
\[
    \widehat{A}
    =\Delta t\,M_y^{-1}\widehat{A}_y.
\]

Moreover, since $\alpha\leq0$ and $A_y\succeq0$ by
\Cref{lem:psd}, we have
\[
    B
    =(1-\Delta t\,\alpha)I_r
      +\Delta t\,(V^\mathfrak{n})^\top A_yV^\mathfrak{n}
    \succ0.
\]
Therefore, all eigenvalues $\lambda_j$ of $B$ satisfy
\[
    \lambda_j>0,
    \qquad j=1,\ldots,r.
\]

By \eqref{eq:shifted-periodic-decomp}, each $\widehat{A}_j$ is a
positive shift of $\widehat{A}$, namely,
\[
    \widehat{A}_j=\widehat{A}+\lambda_j I.
\]
Thus, if $\mu$ is any eigenvalue of $\widehat{A}$, then
$\mu\geq0$, and the corresponding eigenvalue of $\widehat{A}_j$ is
\[
    \mu+\lambda_j>0.
\]
Hence, $\widehat{A}_j$ is nonsingular for every $j=1,\ldots,r$.

Finally, by \eqref{eq:shifted-periodic-decomp},
\[
    \widehat{A}_j+\widehat{U}\widehat{V}^{\top}
    =A+\lambda_j I.
\]
Since $M_x\succ0$ and $A_x\succeq0$ by \Cref{lem:psd},
\[
    M_x^{-1}A_x
    \sim
    M_x^{-1/2}A_xM_x^{-1/2}
    \succeq0.
\]
Therefore, all eigenvalues of $A$ are real and nonnegative. The same
conclusion follows from
\[
    A=\Delta t\,M_y^{-1}A_y
\]
using $M_y\succ0$ and $A_y\succeq0$. Consequently, all eigenvalues of
\[
    A+\lambda_j I
    =\widehat{A}_j+\widehat{U}\widehat{V}^{\top}
\]
are strictly positive, since $\lambda_j>0$. Hence,
$\widehat{A}_j+\widehat{U}\widehat{V}^{\top}$ is nonsingular for every
$j=1,\ldots,r$.

(ii)
By (i), both $\widehat{A}_j$ and
$\widehat{A}_j+\widehat{U}\widehat{V}^{\top}$ are nonsingular. Since
$\widehat{A}_j$ is nonsingular,
\[
\begin{aligned}
\det\left(
    \widehat{A}_j+\widehat{U}\widehat{V}^{\top}
\right)
&=
% \det\left[
%     \widehat{A}_j
%     \left(
%         I_m+\widehat{A}_j^{-1}
%         \widehat{U}\widehat{V}^{\top}
%     \right)
% \right] =
\det(\widehat{A}_j)
\det\left(
    I_m+\widehat{A}_j^{-1}
    \widehat{U}\widehat{V}^{\top}
\right).
\end{aligned}
\]
By Sylvester's determinant theorem,
\[
\det(I_m+XY)=\det(I_2+YX),
\]
with
\[
X=\widehat{A}_j^{-1}\widehat{U}
\in\mathbb{R}^{m\times2},
\qquad
Y=\widehat{V}^{\top}
\in\mathbb{R}^{2\times m}.
\]
Hence,
\[
\det\left(
    I_m+\widehat{A}_j^{-1}
    \widehat{U}\widehat{V}^{\top}
\right)
=
\det\left(
    I_2+\widehat{V}^{\top}
    \widehat{A}_j^{-1}\widehat{U}
\right).
\]
Consequently,
\[
\det\left(
    \widehat{A}_j+\widehat{U}\widehat{V}^{\top}
\right)
=
\det(\widehat{A}_j)
\det\left(
    I_2+\widehat{V}^{\top}
    \widehat{A}_j^{-1}\widehat{U}
\right).
\]
By (i), the determinants of
$\widehat{A}_j$ and
$\widehat{A}_j+\widehat{U}\widehat{V}^{\top}$ are both nonzero.
Therefore,
\[
    \det\left(
        I_2+\widehat{V}^{\top}
        \widehat{A}_j^{-1}\widehat{U}
    \right)\neq0.
\]
Thus, the conclusion (ii) holds.
% \[
%     I_2+\widehat{V}^{\top}\widehat{A}_j^{-1}\widehat{U}
% \]
% is nonsingular.
\end{proof}

Based on \Cref{Ajnons}, we have the following result.
\begin{lemma}\label{lem:smw}
For matrix $\widehat{A}_j+\widehat{U}\widehat{V}^{\top}$ in \eqref{eq:shifted-periodic-decomp},  it holds the Sherman--Morrison--Woodbury \cite{woodbury1950inverting} identity,
\begin{equation}
    \left(
        \widehat{A}_j+\widehat{U}\widehat{V}^{\top}
    \right)^{-1}
    =
    \widehat{A}_j^{-1}
    -
    \widehat{A}_j^{-1}\widehat{U}
    \left(
        I_2+\widehat{V}^{\top}\widehat{A}_j^{-1}\widehat{U}
    \right)^{-1}
    \widehat{V}^{\top}\widehat{A}_j^{-1}.
    \label{eq:woodbury-identity}
\end{equation}
\end{lemma}

By \Cref{lem:smw}, the solution of \eqref{eq:shifted-systemperiod} can be given analytically by
\begin{equation}
\begin{aligned}
    y_j
    &=
    \widehat{A}_j^{-1}\widetilde{c}_j
    -
    \widehat{A}_j^{-1}\widehat{U}
    \left(
        I_2+\widehat{V}^{\top}\widehat{A}_j^{-1}\widehat{U}
    \right)^{-1}
    \widehat{V}^{\top}\widehat{A}_j^{-1}\widetilde{c}_j,
    \qquad j=1,\ldots,r.
\end{aligned}
\label{eq:step3sol}
\end{equation}

In practice, $\widehat{A}_j^{-1}$ is not formed explicitly, since it
would generally be dense. Instead, all required operations are
performed using banded linear solves. Specifically, for each
$j=1,\ldots,r$, we first solve
\begin{equation}
    \widehat{A}_j\widehat{y}_j=\widetilde{c}_j,
    \qquad
    \widehat{y}_j\in\mathbb{R}^m,
    \label{eq:banded-rhs}
\end{equation}
and then solve the matrix-valued banded system
\begin{equation}
    \widehat{A}_j\widehat{Z}_j=\widehat{U},
    \qquad
    \widehat{Z}_j\in\mathbb{R}^{m\times 2}.
    \label{eq:banded-U}
\end{equation}
The two columns of $\widehat{Z}_j$ can be obtained simultaneously
using the same banded factorization of $\widehat{A}_j$. The solution
$y_j$ is then computed as
\begin{equation}
    y_j
    =
    \widehat{y}_j
    -
    \widehat{Z}_j
    \left(
        I_2+\widehat{V}^{\top}\widehat{Z}_j
    \right)^{-1}
    \widehat{V}^{\top}\widehat{y}_j.
    \label{eq:woodbury-correction}
\end{equation}
Thus, the periodic shifted system is reduced to banded linear solves
with $\widehat{A}_j$ plus a low-rank correction.
Repeating this procedure for $j=1,\ldots,r$ yields
$Y=[y_1,\ldots,y_r]$. The solution of the original Sylvester equation
is then recovered by the same back transformation (\textbf{Step 4})
\(
    X=YQ^\top.
\)

% We next show that the Sherman--Morrison--Woodbury identity \eqref{eq:woodbury-identity} holds.

The implementation of the low-cost solver for periodic BCs is given in \Cref{alg:sylvester-periodic}.

\begin{algorithm}[htbp]
\caption{Low-cost solver for \eqref{eq:nd_sylvester} under periodic
BCs.}
\label{alg:sylvester-periodic}
\begin{algorithmic}[1]
\REQUIRE
$A\in\mathbb{R}^{m\times m}$ with the periodic structure
%$A=\widehat{A}+\widehat{U}\widehat{V}^{\top}$, 
\eqref{eq:A-periodic-structure},
where $\widehat{A}$ is banded with half-bandwidth $k$ and
$\widehat{U},\widehat{V}\in\mathbb{R}^{m\times 2}$;
$B=B^\top\in\mathbb{R}^{r\times r}$;
$C\in\mathbb{R}^{m\times r}$.
\ENSURE
$X\in\mathbb{R}^{m\times r}$ satisfying $AX+XB=C$.

\algcomment{\textbf{Step 1}: Eigendecomposition of $B$:}
\STATE Compute the real orthogonal eigendecomposition
$B=Q\Lambda Q^\top$, where
$\Lambda=\operatorname{diag}(\lambda_1,\ldots,\lambda_r)$
and $Q^\top Q=I_r$.

\algcomment{\textbf{Step 2}: Diagonalization of the Sylvester equation:}
\STATE Set
$\widetilde{C}=CQ$.
\STATE Let $Y=XQ$, so that
$AY+Y\Lambda=\widetilde{C}$.

\algcomment{\textbf{Step 3}: Woodbury-corrected shifted solves:}
\FOR{$j=1,\ldots,r$}
    \STATE Form the banded matrix
    $\widehat{A}_j=\widehat{A}+\lambda_j I$.

    \STATE Compute a banded factorization of $\widehat{A}_j$.

    \STATE Solve
    $\widehat{A}_j\widehat{y}_j=\widetilde{c}_j$.

    \STATE Solve \(
        \widehat{Z}_j\in\mathbb{R}^{m\times 2}
    \) from
    \(
        \widehat{A}_j\widehat{Z}_j=\widehat{U}.
        \)

    \STATE Form the matrix
    \(
        R_j=I_2+\widehat{V}^{\top}\widehat{Z}_j \in \mathbb{R}^{2\times2}.
    \)

    \STATE Compute
    \(
        g_j=R_j^{-1}\widehat{V}^{\top}\widehat{y}_j.
    \)

    \STATE Compute
    \(
        y_j=\widehat{y}_j-\widehat{Z}_j g_j.
    \)
\ENDFOR

\STATE Assemble
$Y=[y_1,\ldots,y_r]$.

\algcomment{\textbf{Step 4}: Back transformation:}
\STATE Recover the solution by
$X=YQ^\top$.

\end{algorithmic}
\end{algorithm}

\begin{remark}
    \Cref{alg:sylvester-periodic} applies to mixed BCs with a periodic BC in one spatial direction. The homogeneous Dirichlet and Neumann cases are also included as special cases. In these cases, $\widehat{U}\widehat{V}^{\top}=0$ in \eqref{eq:shifted-periodic-decomp}, so that the correction term in \eqref{eq:woodbury-correction} vanishes.
\end{remark}

The computational complexity of \Cref{alg:sylvester-periodic} is mainly
determined by the four steps. Since $B\in\mathbb{R}^{r\times r}$ is
symmetric, its eigendecomposition costs $O(r^3)$. The two
matrix multiplications $\widetilde{C}=CQ$ and $X=YQ^\top$ each require
$O(mr^2)$ operations. In Step~3, for each eigenvalue
$\lambda_j$, the shifted matrix
$\widehat{A}_j=\widehat{A}+\lambda_j I$ has the same half-bandwidth $k$
as $\widehat{A}$ and can therefore be factorized using banded Gaussian
elimination in $O(mk^2)$ operations. The corresponding
systems with one and two right-hand sides require
$O(mk)$ operations for the forward and backward substitutions.
The formation of
$I_2+\widehat{V}^{\top}\widehat{Z}_j$ requires $O(m)$
additional operations, while the subsequent solution of the
$2\times2$ linear system requires $O(1)$ operations. Since the
$r$ shifted systems are independent, their total cost is
$O(rmk^2)$, with $O(rmk)$ additional work for the
triangular solves and Woodbury corrections. The back transformation
$X=YQ^\top$ in Step~4 requires $O(mr^2)$ operations.
Consequently, the overall computational complexity of
\Cref{alg:sylvester-periodic} is
\(
    O\!\left(r^3+mr^2+mrk^2\right),
\)
where the lower-order $O(rmk)$ and $O(rm)$ terms
have been omitted. In particular, since $k$ is
typically fixed and $r\ll m$, the overall computational complexity is
dominated by $O(mr^2)$. 

Regarding the storage cost, the banded matrix $\widehat{A}$ requires
$O(mk)$ storage, while the Woodbury factors
$\widehat{U},\widehat{V}\in\mathbb{R}^{m\times2}$ require
$O(m)$ storage. The symmetric matrix $B$ and its eigenvector
matrix $Q$ require $O(r^2)$ storage. The matrices
$C$, $Y$, and $X$ each require $O(mr)$ storage. Since the
$r$ shifted banded systems in Step~3 can be solved sequentially, only
one banded LU factorization needs to be stored at a time, requiring
$O(mk)$ additional storage. The intermediate matrix
$\widehat{Z}_j\in\mathbb{R}^{m\times2}$ requires $O(m)$
storage. Hence, the overall storage requirement is
\(
    O(mk+mr+r^2).
\)
In particular, since $k$ is typically fixed and $r\ll m$, the storage
cost is essentially $O(mr)$.

\subsection{Computational and storage costs}\label{sec:cost}

Finally, we summarize the computational and storage costs of the three
approaches considered in this work. For clarity, we classify the methods as follows:
\begin{itemize}
    \item \textbf{Approach I:} FR-MLFEM with the conventional Sylvester solver;
    
    \item \textbf{Approach II:} BUG-MLFEM with the conventional Sylvester solver;

    \item \textbf{Approach III (Proposed):} BUG-MLFEM utilizing the proposed low-cost non-iterative Sylvester solver.
    The $K/L$-equations are solved via \Cref{alg:sylvester-sym} (Dirichlet/Neumann) and \Cref{alg:sylvester-periodic} (periodic), while the $S$-equation uses the conventional Sylvester solver.
\end{itemize}

In Approach~I, the full-rank matrix equation \eqref{frfemmat} arising from
FR-MLFEM is solved using the Bartels--Stewart algorithm
\cite{bartels1972algorithm}. In Approach~II, the same algorithm is applied
to the matrix equations arising in \Cref{alg:KLS-Steps}, yielding the
conventional Sylvester solver. In Approach~III, these equations are solved
using the proposed low-cost, non-iterative Sylvester solvers. The
computational and storage costs per time step for the three approaches
are summarized in \Cref{tab:cost-storage-comparison}.
As shown in \Cref{tab:cost-storage-comparison}, Approach~III substantially
reduces both computational and storage costs relative to Approaches~I
and~II. Specifically, its computational and storage complexities per
time step are
\(
O\bigl((m+n)(r^2+rk^2)+r^3\bigr)\)
{and} \(
O\bigl((m+n)(k+r)+r^2\bigr),
\)
respectively. When the polynomial degree $k$ is fixed and the rank
satisfies $r\ll m,n$, these bounds simplify to
$O((m+n)r^2)$ for computation and $O((m+n)r)$ for storage, highlighting
the efficiency of the proposed approach in the low-rank regime.

\begin{table}[H]
\centering
\caption{Computational and storage costs per time step.}
\label{tab:cost-storage-comparison}
\begin{tabular}{@{}lll@{}}
\toprule
\textbf{Approach} & \textbf{Computational cost} & \textbf{Storage cost} \\
\midrule
Approach I
    & $O(m^2n+mn^2)$
    & $O(m^2+n^2)$ \\
Approach II
    & $O(m^3+n^3)$
    & $O(m^2+n^2)$ \\
Approach III
    & $O((m+n)r^2)$
    & $O((m+n)r)$ \\
\bottomrule
\end{tabular}
\end{table}

\section{Numerical Experiments}
\label{sec:experiments}
We assess the accuracy, convergence behavior, and computational efficiency of the proposed BUG-MLFEM with low-cost, non-iterative Sylvester solvers. Unless otherwise stated, all numerical experiments are based on the two-dimensional reaction-diffusion equation \eqref{diffusion} with $\alpha=-1$ on the unit square $\Omega=[0,1]^2$ over the time interval $[0,T]$, with $T=0.1$. We consider three types of boundary conditions: homogeneous Dirichlet, homogeneous Neumann, and periodic. For each type of boundary condition, the spatial discretization employs $P_k$ finite element spaces with polynomial degrees $k=1,2,3$ with the grid size $M=N$. In time, we use the backward Euler method, resulting in first-order implicit FR-MLFEM and BUG-MLFEM schemes. The BUG-MLFEM scheme is equivalent to the first-order implicit BUG integrator for WDLRA. We use the time step size
\(
\Delta t = 0.05h^{k+1}
\)
and fix the low rank at $r=10$ in all low-rank computations.
%The initial low-rank factors $U^0$, $S^0$, and $V^0$ are obtained by applying the GSVD in \Cref{alg:gsvd} to the initial coefficient matrix $W^0$.
The numerical tests are conducted on a desktop equipped with an Intel\textsuperscript{\textregistered} Core\textsuperscript{\texttrademark} Ultra 7 265 processor and 64~GB of RAM.

To examine the solvers' applicability for the boundary conditions, we consider the following three test cases, corresponding to homogeneous Dirichlet,
homogeneous Neumann, and periodic BCs. In each case,
the initial condition is selected such that an exact solution is
available.

%-------------------------------------------------------------------------------
\noindent{\textbf{Test Case 1: Homogeneous Dirichlet BC.}}
%-------------------------------------------------------------------------------
We consider the homogeneous Dirichlet BC, with the initial condition chosen as $w(x,y,0)
    =
    \sin(\pi x)\sin(\pi y),$
which gives the exact solution
\begin{equation}
    w(x,y,t)
    =
    e^{-(2\pi^2+1)t}
    \sin(\pi x)\sin(\pi y).
    \label{eq:dbc_exact}
\end{equation}

%-------------------------------------------------------------------------------
\noindent{\textbf{Test Case 2: Homogeneous Neumann BC.}}
%-------------------------------------------------------------------------------
We consider the homogeneous Neumann BC, with 
the initial condition taken as $w(x,y,0)
    =
    \cos(\pi x)\cos(\pi y),$
with the corresponding exact solution
\begin{equation}
    w(x,y,t)
    =
    e^{-(2\pi^2+1)t}
    \cos(\pi x)\cos(\pi y).
    \label{eq:nbc_exact}
\end{equation}

%-------------------------------------------------------------------------------
\noindent{\textbf{Test Case 3: Periodic BCs.}}
%-------------------------------------------------------------------------------
Finally, we consider periodic BCs in both spatial
directions, with
the initial condition chosen as $ w(x,y,0)
    =
    \sin(2\pi x)\cos(2\pi y),$
and the exact solution is
\begin{equation}
    w(x,y,t)
    =
    e^{-(8\pi^2+1)t}
    \sin(2\pi x)\cos(2\pi y).
    \label{eq:pbc_exact}
\end{equation}

\begin{example}[Accuracy and convergence tests]
\label{ex:accuracy}
In this example, we assess the spatial and temporal convergence
properties of the FR-MLFEM and BUG-MLFEM schemes.

% %-------------------------------------------------------------------------------
% \paragraph{Spatial convergence}
% %-------------------------------------------------------------------------------

We first investigate the spatial convergence of the FR-MLFEM and
BUG-MLFEM schemes. The spatial
discretization uses $P_k$ finite elements with $k=1,2,3$, and the
error is measured in the mass-lumped $L^2(\Omega)$ norm at the final time
$T=0.1$.
% \begin{equation}
%     e_h
%     =
%     \left\|
%         w_h(\cdot,T)-w(\cdot,T)
%     \right\|_{L^2(\Omega)}.
%     \label{eq:l2_error}
% \end{equation}
For the spatial convergence tests, we take the default time step size. 
%The spatial mesh is uniformly refined, with the mesh size $h$ successively reduced by a factor of two.
The spatial convergence results are presented in
\Cref{fig:spatial_accuracy_all_bcs}. For both FR-MLFEM and BUG-MLFEM,
the numerical errors decrease at the optimal rate $k+1$ for the
$P_k$ finite element approximations. More importantly, the optimal
spatial convergence behavior is consistently maintained for all three
types of boundary conditions.
The BUG-MLFEM scheme exhibits the same asymptotic convergence rates as
the corresponding FR-MLFEM scheme. This agreement confirms that the
low-rank dynamical formulation does not degrade the spatial accuracy
of the underlying finite element discretization.

\begin{figure}[htbp]
    \centering

    % Top row: FR-MLFEM
    \includegraphics[width=0.315\textwidth]{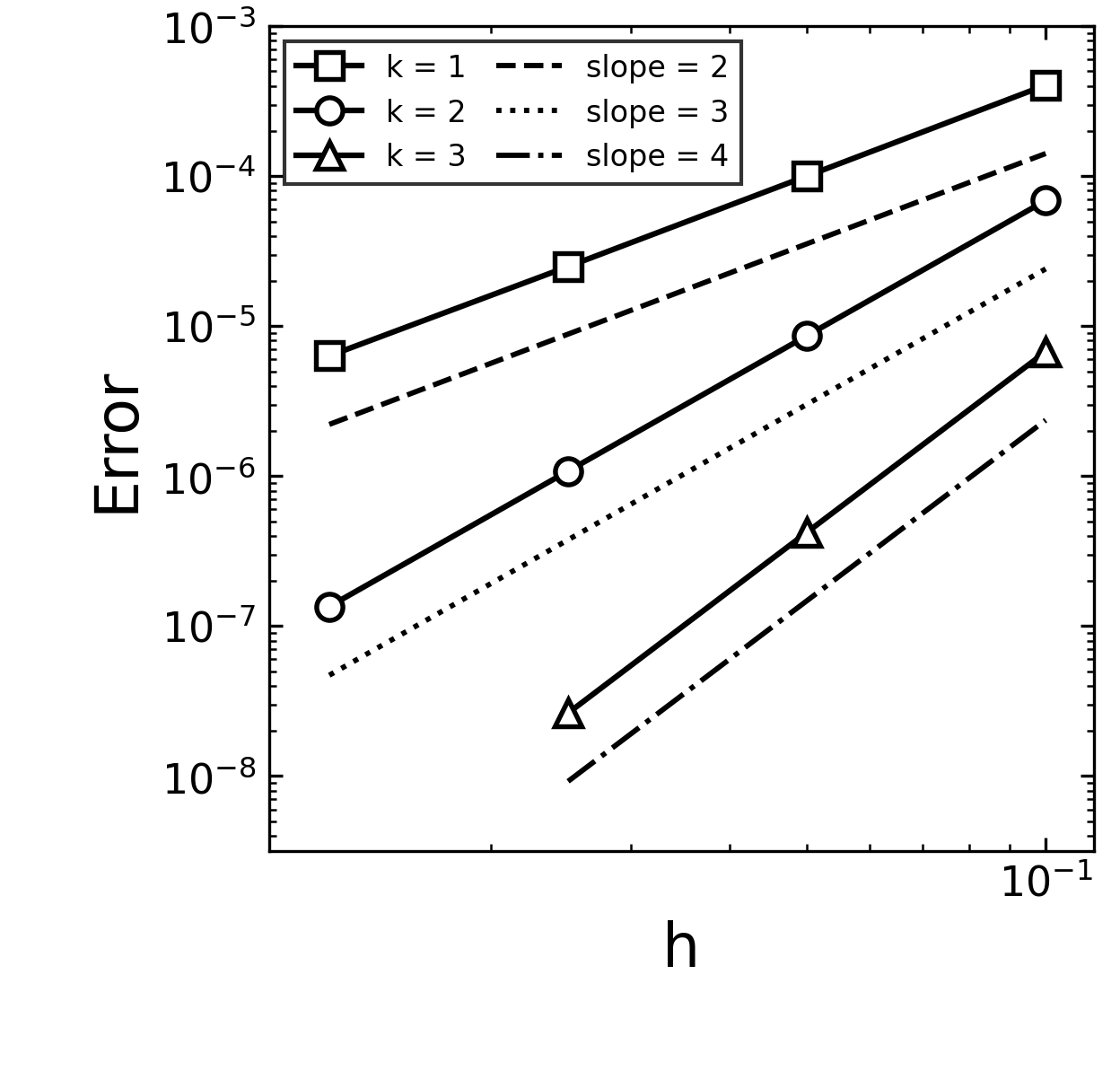}
    \hfill
    \includegraphics[width=0.315\textwidth]{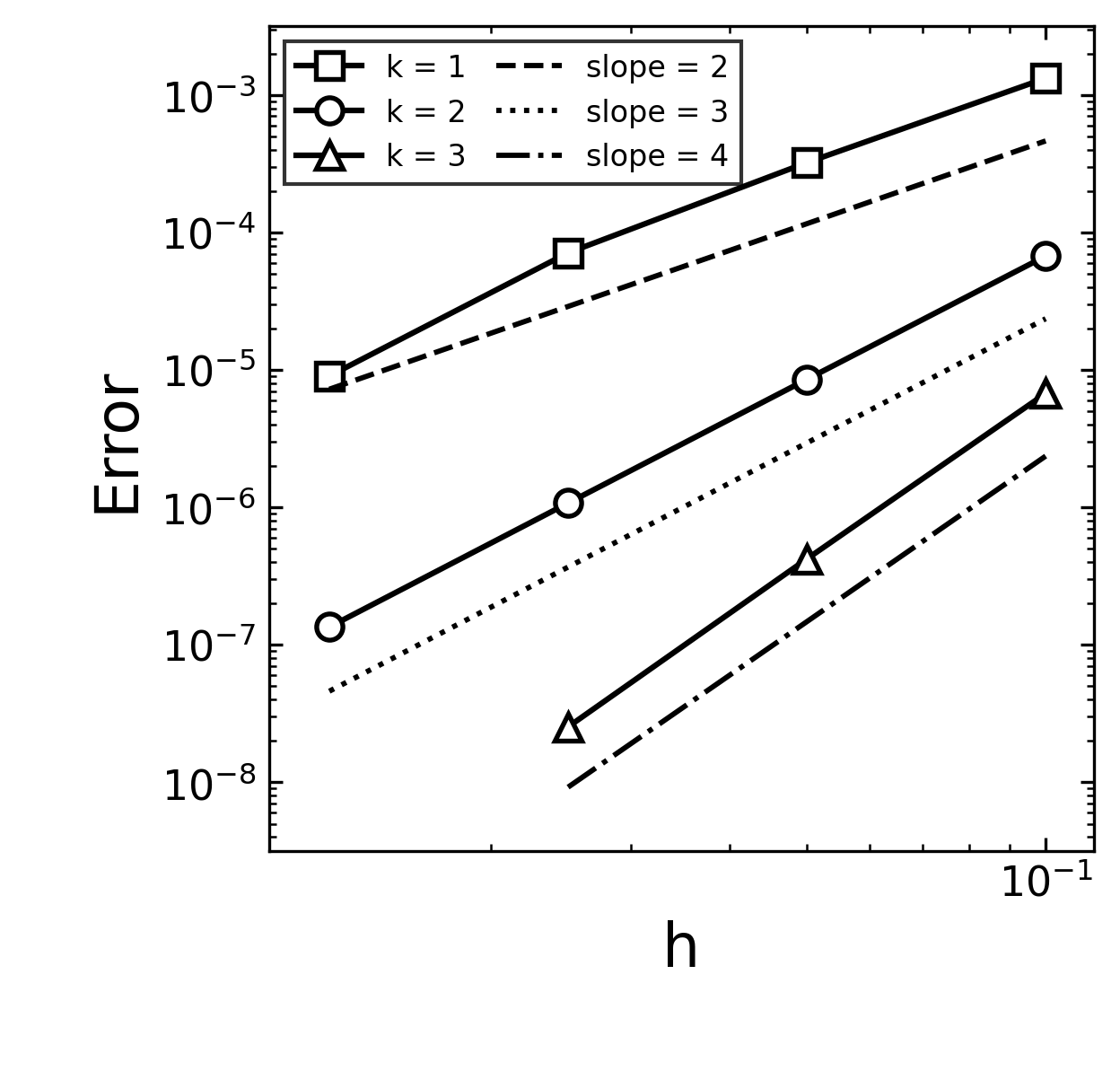}
    \hfill
    \includegraphics[width=0.315\textwidth]{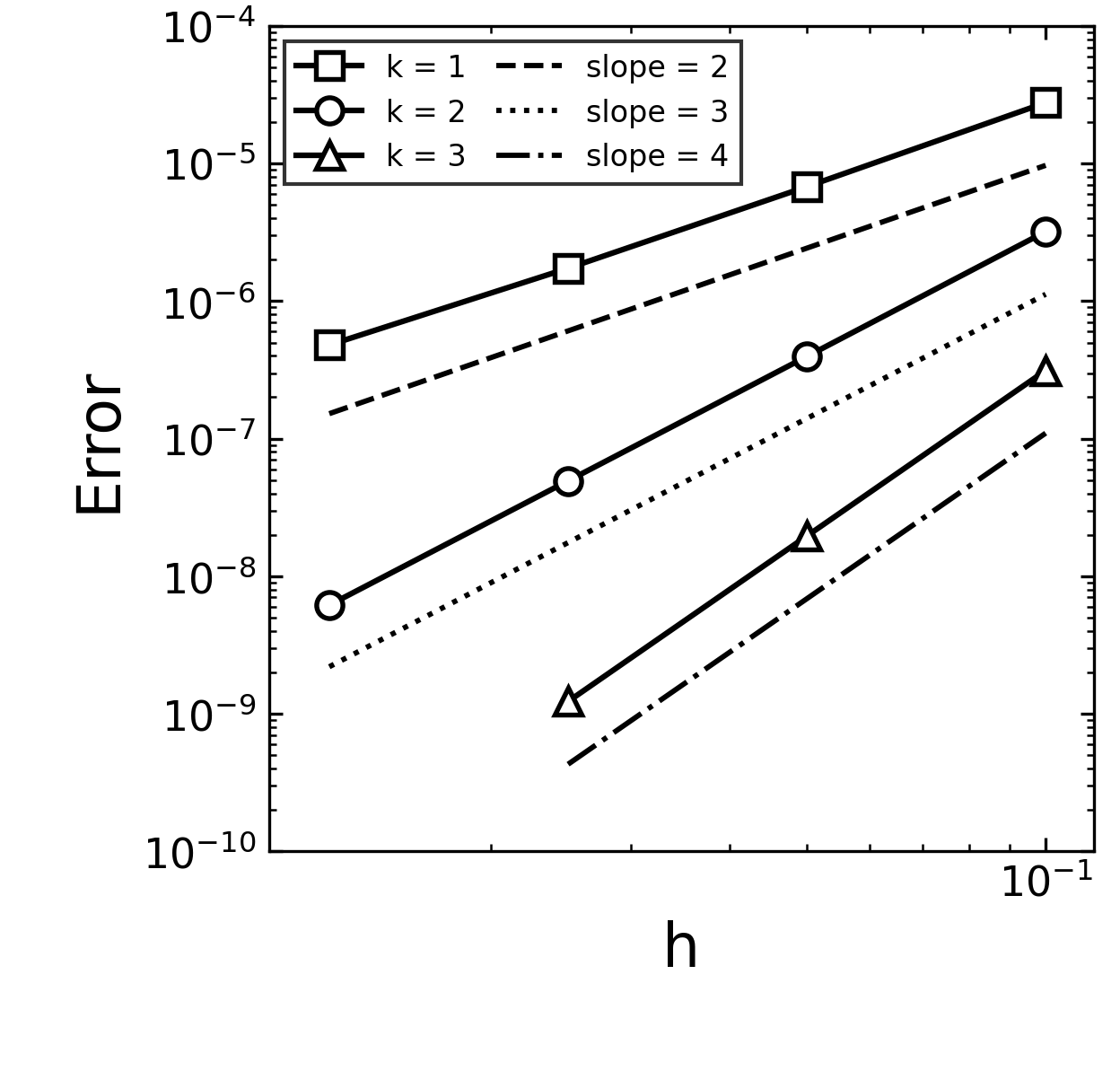}

    \vspace{0.1cm}

    % Bottom row: BUG-MLFEM
    \includegraphics[width=0.315\textwidth]{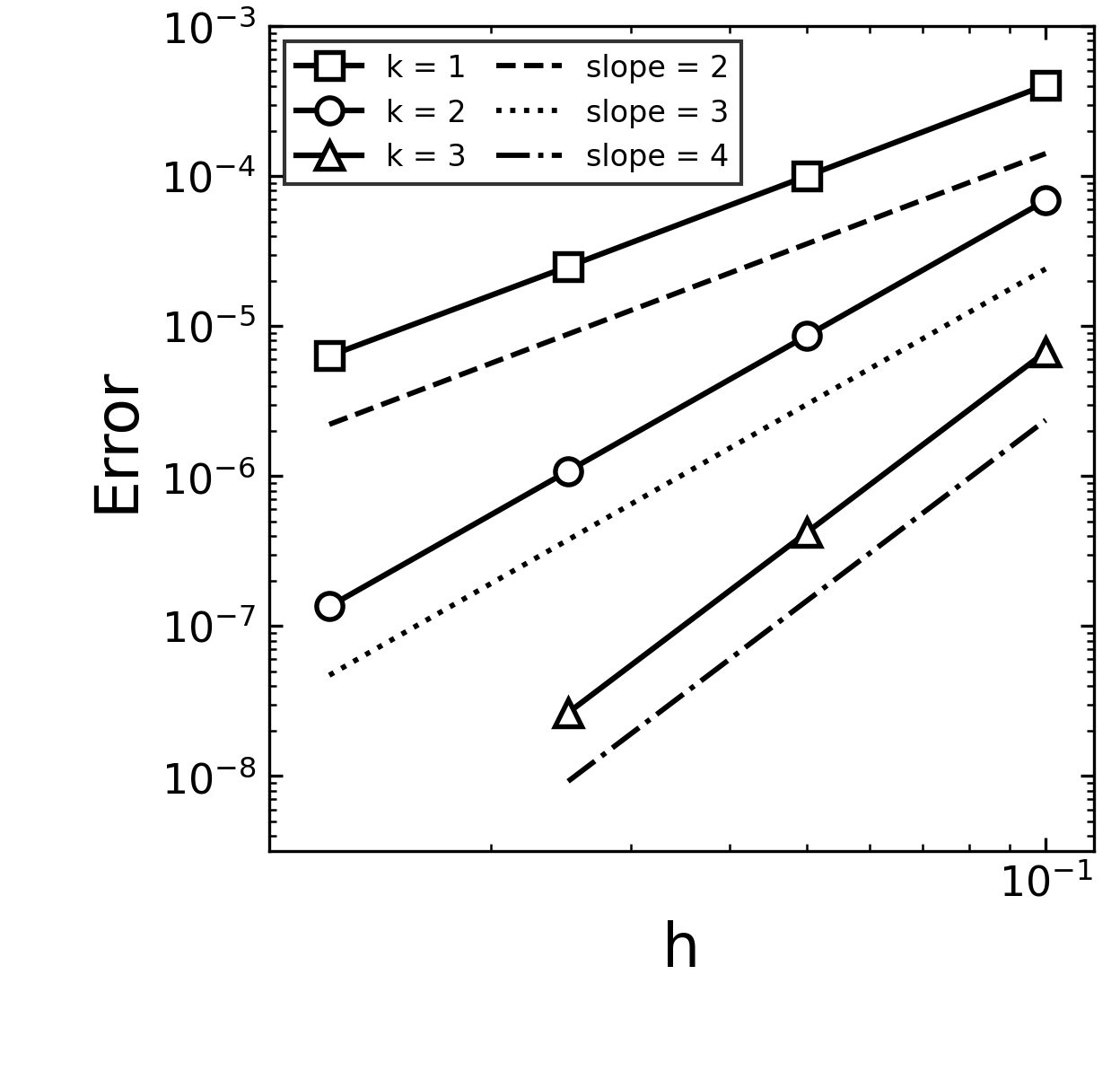}
    \hfill
    \includegraphics[width=0.315\textwidth]{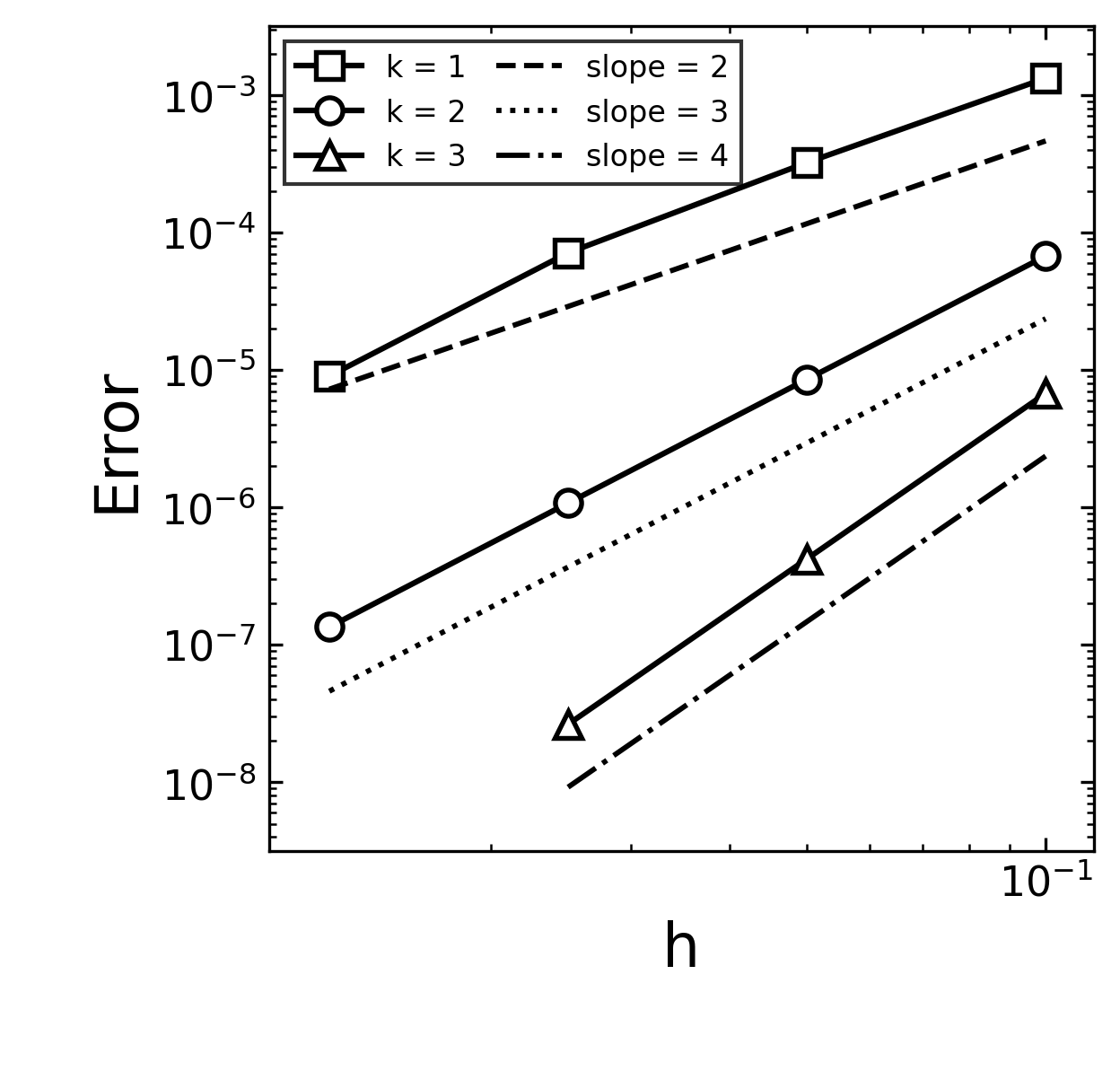}
    \hfill
    \includegraphics[width=0.315\textwidth]{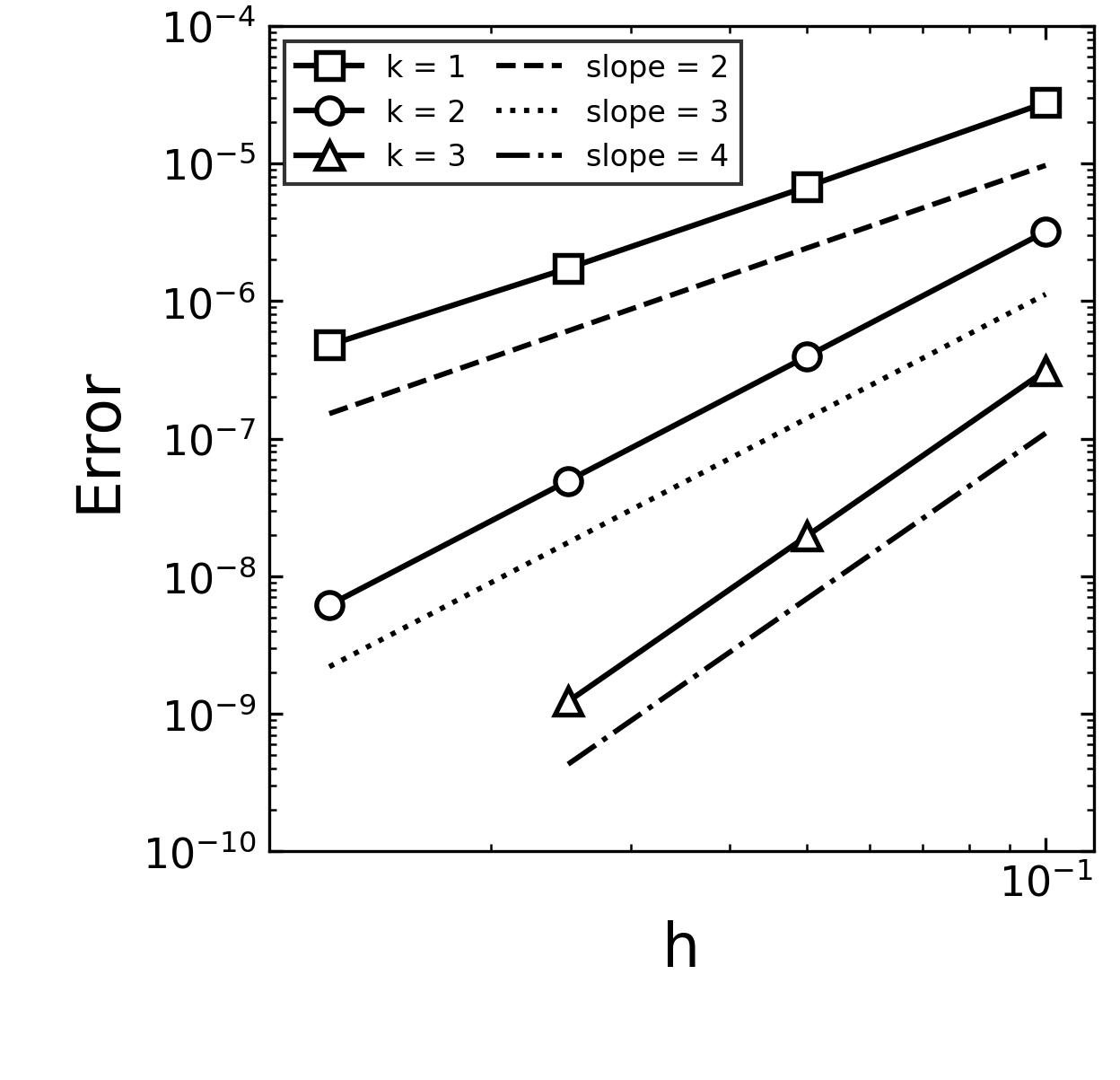}

    \caption{Spatial convergence of the FR-MLFEM (top row) and
BUG-MLFEM (bottom row) schemes using $P_k$ elements with
$k=1,2,3$ at $T=0.1$. From left to right, the columns correspond to
homogeneous Dirichlet, homogeneous Neumann, and periodic boundary
conditions.}
    \label{fig:spatial_accuracy_all_bcs}
\end{figure}

% %-------------------------------------------------------------------------------
% \paragraph{Temporal convergence}
% %-------------------------------------------------------------------------------

We next investigate the temporal convergence of the FR-MLFEM and
BUG-MLFEM schemes. We
fix the spatial mesh sufficiently fine and use $P_2$ finite
elements. 
\iffalse
The time step is successively refined according to
\(
    \Delta t,\
    \frac{\Delta t}{2},\
    \frac{\Delta t}{4},\
    \frac{\Delta t}{8},\ 
    \frac{\Delta t}{16}
\) with $\Delta t =0.01$,
while the final time remains fixed at $T=0.1$. 
\fi
The error is again
measured in the mass-lumped $L^2(\Omega)$ norm at the final time.
The temporal convergence results are shown in
\Cref{fig:temporal_convergence}. Both FR-MLFEM and BUG-MLFEM exhibit
the expected first-order temporal convergence. The same first-order
behavior is observed for all three types of boundary conditions,
confirming that the low-rank formulation preserves the temporal
accuracy of the underlying first-order time integrator.

\begin{figure}[htbp]
    \centering
    \includegraphics[width=0.32\textwidth]{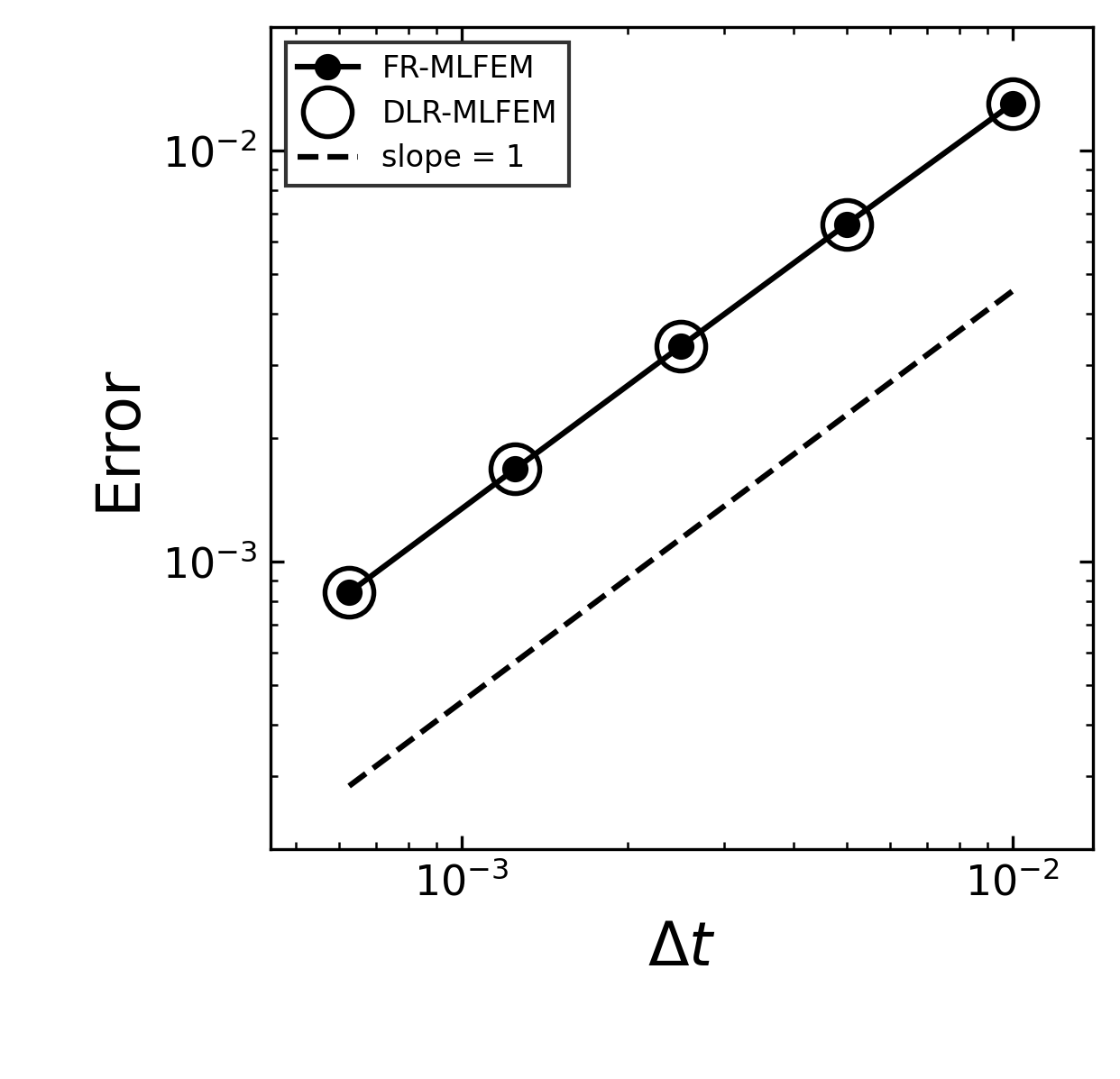}
    \hfill
    \includegraphics[width=0.32\textwidth]{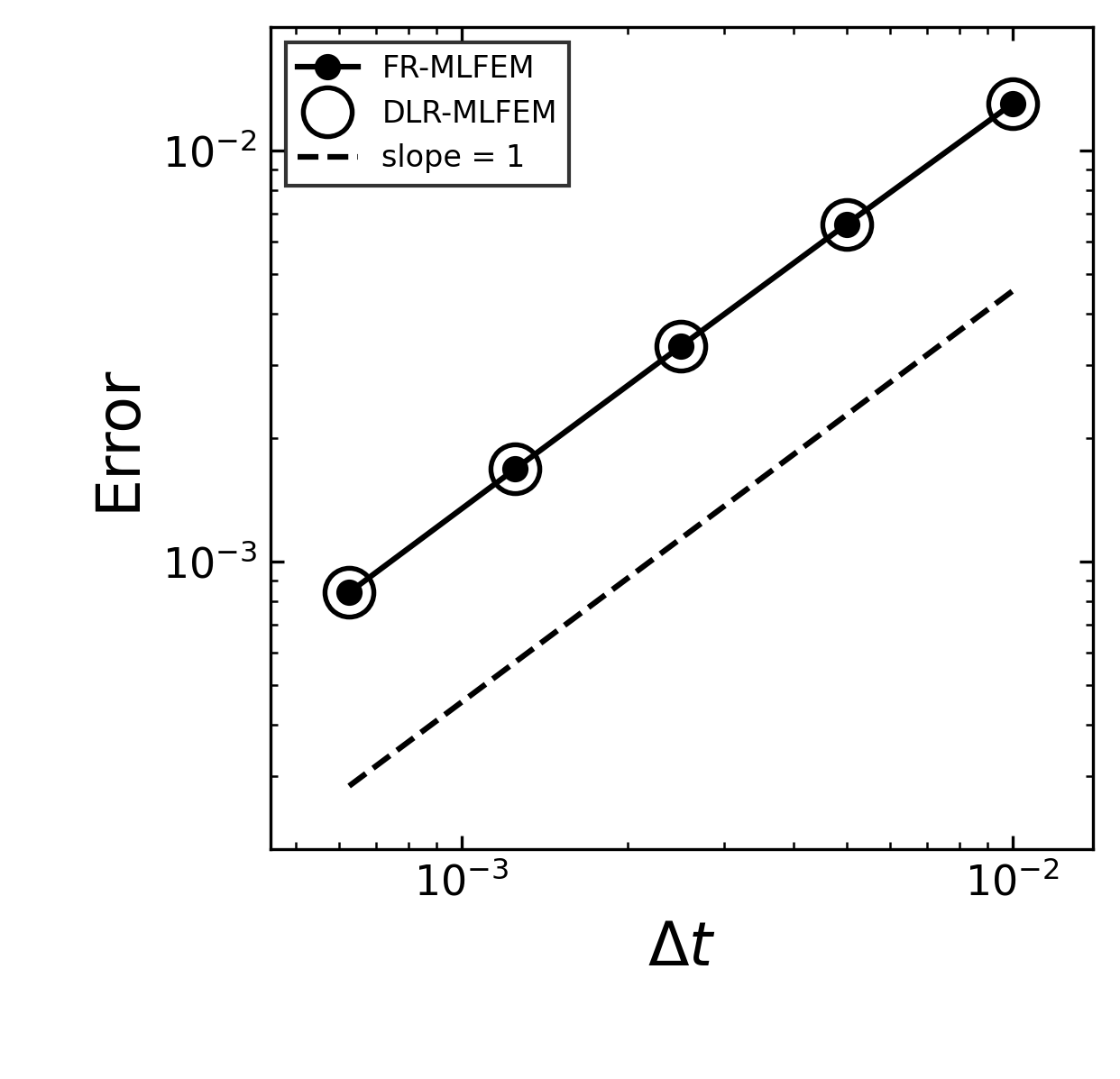}
    \hfill
    \includegraphics[width=0.32\textwidth]{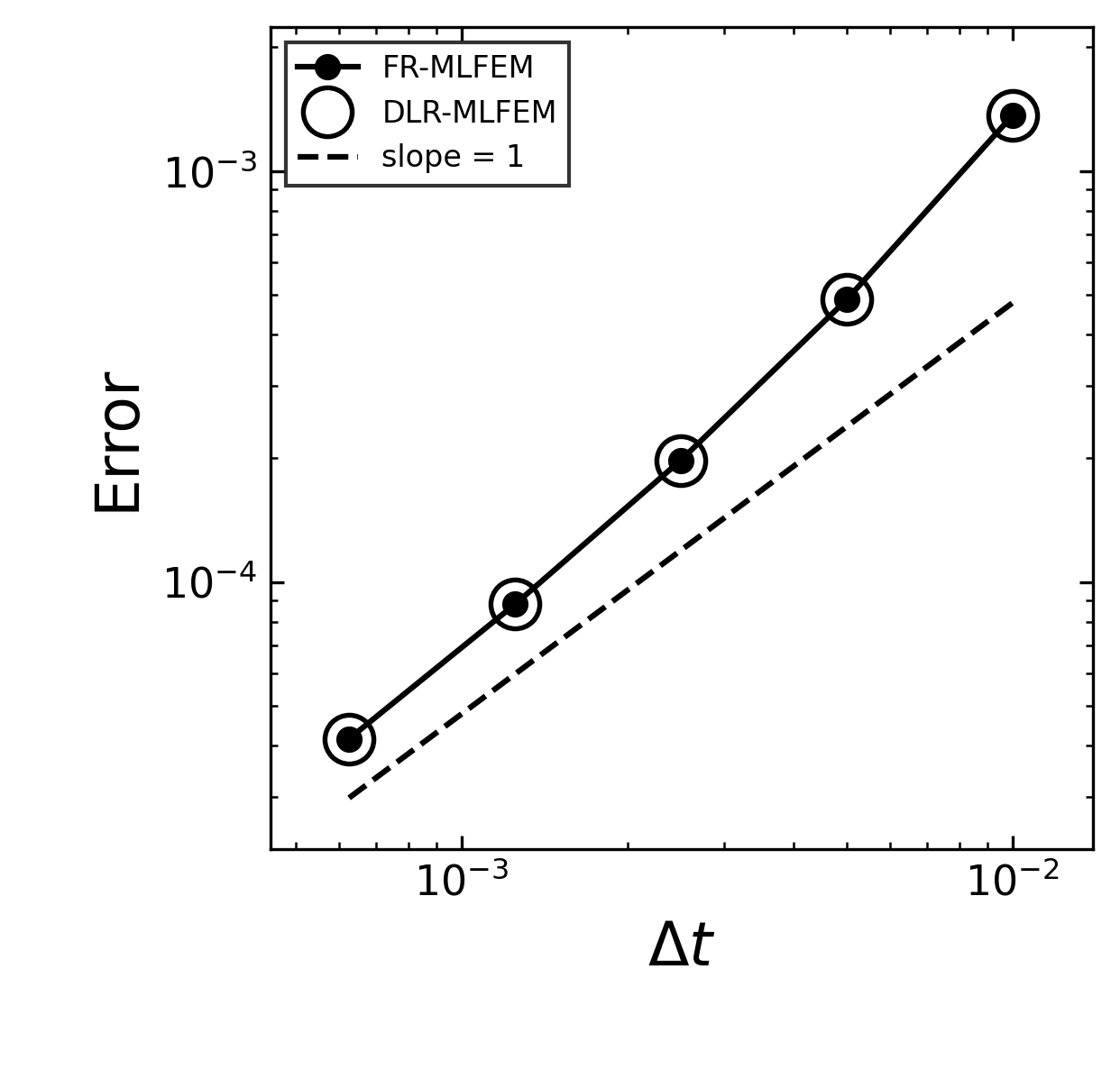}
    \caption{Temporal convergence of the FR-MLFEM and BUG-MLFEM schemes for $P_2$ polynomials. From left to right: homogeneous Dirichlet, homogeneous Neumann, and periodic BCs.}
    \label{fig:temporal_convergence}
\end{figure}

\end{example}

\begin{example}
    We continue the test cases above to compare the computational costs of the FR-MLFEM and BUG-MLFEM schemes using the different approaches described in \Cref{sec:cost}. All tests employ $P_2$ finite elements, the default time step, and a final time of $T=0.01$.
The CPU-time results in \Cref{tab:cpu_all_bc} demonstrate the substantial computational advantage of Approach~III over Approaches~I and~II for all three boundary conditions. For the relatively coarse mesh ($N=40$), Approach~III reduces the CPU time from $74.93$--$126.57$ seconds for Approach~I to $8.14$--$22.70$ seconds, achieving speedups ranging from $5.58\times$ to $9.21\times$. In comparison, Approach~II provides only modest improvements, with speedups of $1.06\times$--$1.12\times$. As the mesh is refined to $N=80$, the computational advantage of Approach~III becomes more pronounced. While Approaches~I and~II require approximately $6.9\times10^3$--$1.2\times10^4$ seconds, Approach~III completes the computations in $178.03$--$471.14$ seconds, yielding speedups of $25.49\times$--$38.80\times$. The advantage is even more significant for the finest mesh, $N=160$, where Approaches~I and~II require approximately $2.1\times10^5$--$3.6\times10^5$ seconds, whereas Approach~III requires only $2838.81$--$6576.58$ seconds. The resulting speedups reach $54.35\times$--$81.38\times$, with the largest gain observed for homogeneous Neumann boundary conditions. These results demonstrate that the computational savings of the proposed Approach~III increase substantially with mesh refinement, whereas Approach~II offers only limited improvements over Approach~I. The speedup for periodic boundary conditions is comparatively smaller than those for homogeneous Dirichlet and Neumann boundary conditions on the finer meshes, which is consistent with the additional low-rank correction required to account for the periodic wrap-around entries in \Cref{alg:sylvester-periodic}.

\begin{table}[h]
\centering
\caption{CPU time and speedup for $P_2$ finite elements. The speedup is measured relative to Approach~I for each boundary condition and mesh size.}
\label{tab:cpu_all_bc}
\small
\setlength{\tabcolsep}{3.5pt}
\begin{adjustbox}{max width=\linewidth}
\begin{tabular}{@{}ll rr rr rr@{}}
\toprule
\multirow{2}{*}{Boundary} & \multirow{2}{*}{Approach}
& \multicolumn{2}{c}{$N=40$}
& \multicolumn{2}{c}{$N=80$}
& \multicolumn{2}{c}{$N=160$} \\
\cmidrule(lr){3-4} \cmidrule(lr){5-6} \cmidrule(l){7-8}
& & Time (s) & Speedup & Time (s) & Speedup & Time (s) & Speedup \\
\midrule
\multirow{3}{*}{Dirichlet}
& I            & 74.93          & 1.00$\times$
               & 6907.73        & 1.00$\times$
               & 212398.83      & 1.00$\times$ \\
& II           & 67.76          & 1.11$\times$
               & 6620.85        & 1.04$\times$
               & 197226.01      & 1.08$\times$ \\
& \textbf{III} & \textbf{8.14}  & \textbf{9.21}$\times$
               & \textbf{178.03} & \textbf{38.80}$\times$
               & \textbf{2838.81} & \textbf{74.82}$\times$ \\
\midrule
\multirow{3}{*}{Neumann}
& I            & 126.57         & 1.00$\times$
               & 8950.84        & 1.00$\times$
               & 288031.73      & 1.00$\times$ \\
& II           & 119.87         & 1.06$\times$
               & 8877.27        & 1.01$\times$
               & 273967.76      & 1.05$\times$ \\
& \textbf{III} & \textbf{22.70}  & \textbf{5.58}$\times$
               & \textbf{252.27} & \textbf{35.48}$\times$
               & \textbf{3539.16} & \textbf{81.38}$\times$ \\
\midrule
\multirow{3}{*}{Periodic}
& I            & 117.10         & 1.00$\times$
               & 12008.88       & 1.00$\times$
               & 357406.11      & 1.00$\times$ \\
& II           & 104.10         & 1.12$\times$
               & 11398.99       & 1.05$\times$
               & 343819.75      & 1.04$\times$ \\
& \textbf{III} & \textbf{20.68} & \textbf{5.66}$\times$
               & \textbf{471.14} & \textbf{25.49}$\times$
               & \textbf{6576.58} & \textbf{54.35}$\times$ \\
\bottomrule
\end{tabular}
\end{adjustbox}
\end{table}

\end{example}

\begin{example}
To further examine the computational cost of the proposed Sylvester
solvers, we investigate the dependence of the CPU time, excluding the initial GSVD of $W^0$, on the spatial grid size $N$ and rank $r$. 
Here, the spatial discretization is based on the \(P_1\) polynomials.
First, we fix $r=2$ and increase $N$ from
$10$ to $320$, with the results reported in
\Cref{tab:cpu_time_N_r2}. For all three types of boundary conditions, the
CPU time remains nearly unchanged for small $N$, where the fixed
computational overhead is dominant. As $N$ increases, the CPU time grows
more rapidly, and the observed scaling rates approach or exceed one at
the finest mesh. In particular, at $N=320$, the CPU times are $97.73$,
$87.12$, and $119.48$ seconds for the Dirichlet, Neumann, and Periodic
BCs, respectively. The results demonstrate that the
computational cost exhibits approximately linear growth with respect to
$N$ in the practically relevant large-$N$ regime, consistent with the
linear dependence on the number of spatial degrees of freedom in the
proposed solver.

\begin{table}[htbp]
\centering
\small
\caption{CPU time (excluding the initial step) for the proposed Sylvester solver as the number of elements $N(=M)$ increases, with fixed rank $r=2$. 
%The observed rate is computed from two consecutive mesh sizes.
}
\label{tab:cpu_time_N_r2}
\setlength{\tabcolsep}{6pt}
\begin{tabular}{@{}c cc cc cc@{}}
\toprule
$N$
& \multicolumn{2}{c}{Dirichlet}
& \multicolumn{2}{c}{Neumann}
& \multicolumn{2}{c}{Periodic} \\
\cmidrule(lr){2-3}
\cmidrule(lr){4-5}
\cmidrule(lr){6-7}
& Time (s) & Rate
& Time (s) & Rate
& Time (s) & Rate \\
\midrule
10  & 19.8036 & ---
    & 18.7370 & ---
    & 35.0732 & --- \\
20  & 20.2239 & 0.0303
    & 19.2211 & 0.0368
    & 35.7823 & 0.0289 \\
40  & 21.5476 & 0.0915
    & 20.2841 & 0.0777
    & 37.4788 & 0.0668 \\
80  & 25.3766 & 0.2360
    & 23.7829 & 0.2296
    & 42.6061 & 0.1850 \\
160 & 40.2759 & 0.6664
    & 36.2545 & 0.6082
    & 57.8492 & 0.4412 \\
320 & 97.7278 & 1.2789
    & 87.1226 & 1.2649
    & 119.4764 & 1.0464 \\
\bottomrule
\end{tabular}
\end{table}

Next, we fix $N=320$ and vary the rank from $r=1$ to $r=10$, as shown in
\Cref{tab:cpu_time_r_N320}. The CPU time increases steadily with $r$ for
all three boundary conditions. For example, for the Dirichlet case, the
CPU time increases from $76.53$ seconds at $r=1$ to $204.42$ seconds at
$r=10$, while the corresponding increase is from $72.39$ to
$200.87$ seconds for the Neumann case. The periodic case requires
additional computational work due to the low-rank correction in
\Cref{alg:sylvester-periodic}, resulting in consistently higher CPU
time. The observed scaling rates with respect to $r$ remain below the
quadratic rate associated with the $O((m+n)r^2)$ complexity. This
is consistent with the overall complexity of the proposed solver, which
also contains lower-order terms in $r$ and rank-independent operations.
Consequently, the measured CPU time grows more slowly than quadratically
over the tested range of $r$.

\begin{table}[htbp]
\centering
\small
\caption{
% CPU time (excluding the initial step) and observed computational scaling with respect to the
% rank $r$ for fixed grid size $N=320$ using $P_1$ finite elements. 
%The observed rate between two consecutive reported values of $r$ is computed by $\log(T_r/T_{r_{\rm prev}})/\log(r/r_{\rm prev})$.
CPU time (excluding the initial step) and observed computational scaling with respect to the rank \(r\) for \(M=N=320\) using \(P_1\) finite elements.
}
\label{tab:cpu_time_r_N320}
\setlength{\tabcolsep}{6pt}
\begin{tabular}{@{}c cc cc cc@{}}
\toprule
& \multicolumn{2}{c}{Homogeneous Dirichlet}
& \multicolumn{2}{c}{Homogeneous Neumann}
& \multicolumn{2}{c}{Periodic} \\
\cmidrule(lr){2-3}
\cmidrule(lr){4-5}
\cmidrule(lr){6-7}
$r$ & Time (s) & Rate
   & Time (s) & Rate
   & Time (s) & Rate \\
\midrule
1  &  76.5336 & ---
   &  72.3940 & ---
   &  87.6317 & --- \\

2  &  95.0596 & 0.3127
   &  88.3810 & 0.2879
   & 113.0220 & 0.3671 \\

4  & 118.7699 & 0.3213
   & 117.2691 & 0.4080
   & 160.5381 & 0.5063 \\

6  & 146.3233 & 0.5145
   & 141.3002 & 0.4598
   & 209.4784 & 0.6563 \\

8  & 176.7193 & 0.6561
   & 172.1958 & 0.6874
   & 261.2554 & 0.7678 \\

10 & 204.4172 & 0.6525
   & 200.8726 & 0.6903
   & 311.0157 & 0.7813 \\
\bottomrule
\end{tabular}
\end{table}

\end{example}

\begin{example}\label{ex:augbug}

Following \cite{ceruti2024robust}, we assess the accuracy, computational
efficiency, and rank adaptivity of the augmented BUG integrator
\Cref{alg:AugBUG-Steps} for a stiff, forced heat equation on the two-dimensional
torus $\Omega=[-\pi,\pi)^2$:
\begin{equation}\label{eq:stiff_heat}
  \partial_t u(x,y,t)
  = \frac{1}{2}\left(
      \partial_{xx}u(x,y,t)+\partial_{yy}u(x,y,t)
    \right)+g(x,y),
\end{equation}
subject to periodic boundary conditions in both spatial directions. The
time-independent forcing term and the initial condition are, respectively, given by
\begin{equation}\label{eq:stiff_forcing}
  g(x,y)=\sum_{j=1}^{11}10^{-(j-1)}
  \exp\!\bigl(-j(x^2+y^2)\bigr)
\end{equation}
and
\begin{equation}\label{eq:stiff_ic}
  u(x,y,0)=\sin(x)\sin(y).
\end{equation}

We first examine spatial convergence. We fix the time step at
$\Delta t=10^{-5}$ and set the initial rank to $r^0=10$. For polynomial
degrees $k=1,2,3$, we refine the spatial mesh and measure the $L^2(\Omega)$
error at $T=1$. Since the exact solution is unavailable, we
use the FR-MLFEM solution on a fine mesh with $N=160$ as the reference
solution. The results are presented in \Cref{fig:stiff_heat_space}.
Both FR-MLFEM and augmented BUG-MLFEM (ABUG-MLFEM) exhibit the optimal spatial convergence
rate $k+1$. Moreover, their error curves are nearly indistinguishable,
suggesting that the low-rank approximation
introduces no discernible error beyond that associated with spatial
discretization.

\begin{figure}[htbp]
  \centering
  \begin{subfigure}[b]{0.45\textwidth}
    \centering
    \includegraphics[width=\textwidth]{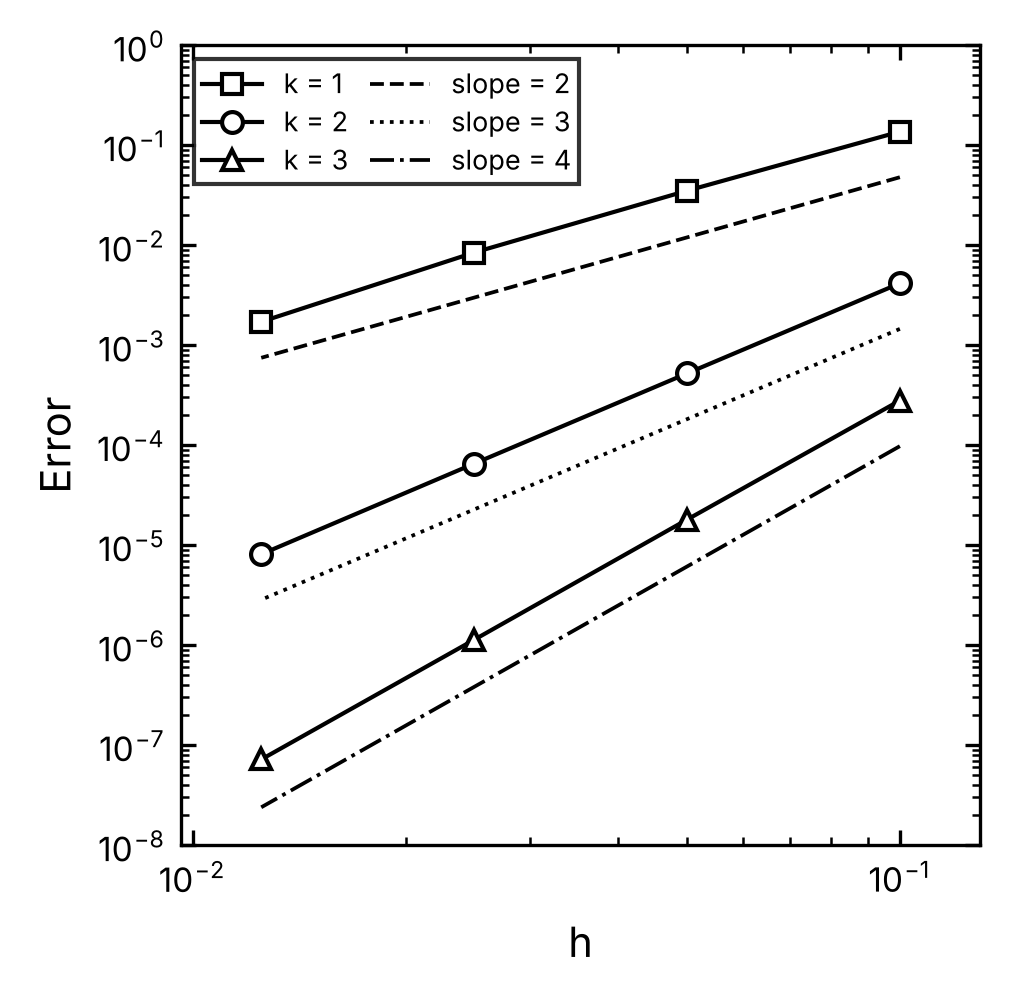}
    \label{fig:stiff_heat_space_FR}
  \end{subfigure}
  \hfill
  \begin{subfigure}[b]{0.45\textwidth}
    \centering
    \includegraphics[width=\textwidth]{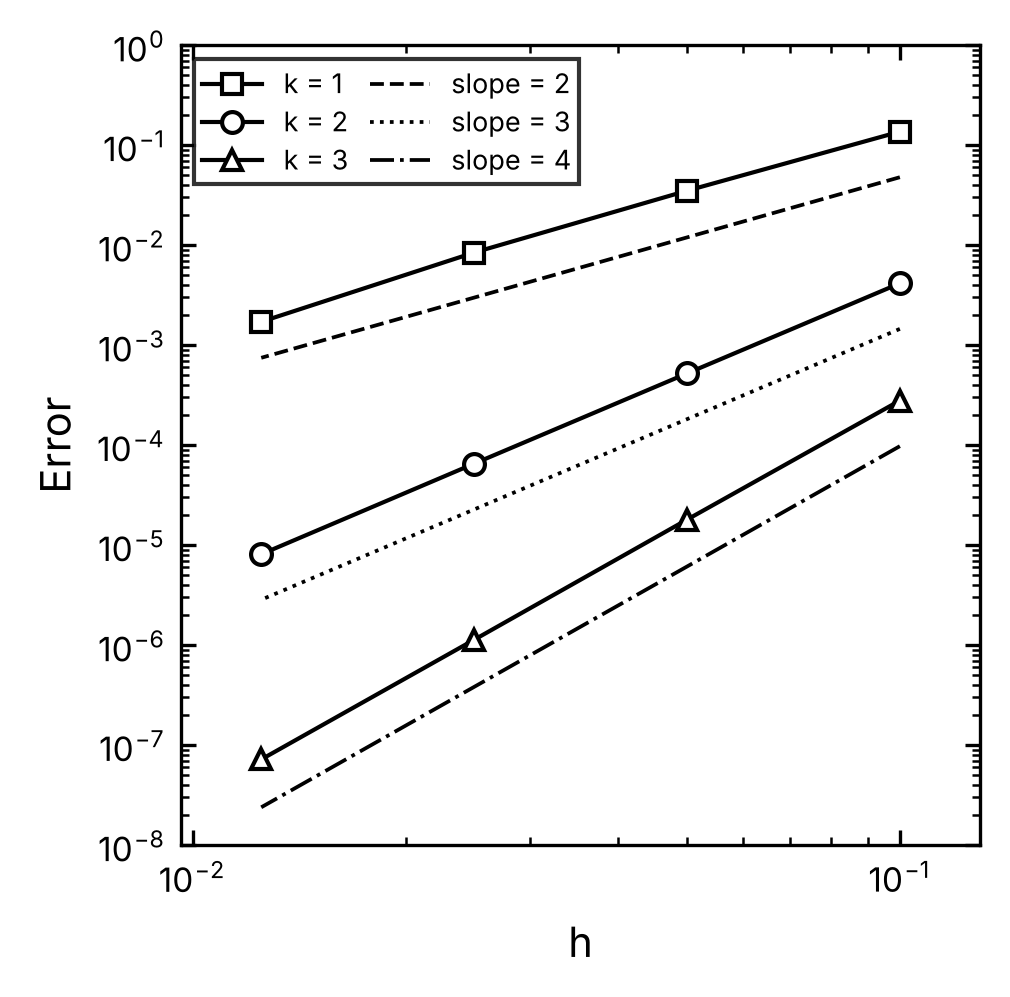}
    \label{fig:stiff_heat_space_BUG}
  \end{subfigure}
  \caption{Spatial convergence at $T=1$ for polynomial degrees
$k=1,2,3$, with $\Delta t=10^{-5}$ and initial rank $r^0=10$.
Left: FR-MLFEM; right: ABUG-MLFEM.
  %The reference lines have slopes $2$, $3$, and $4$.
  }
  \label{fig:stiff_heat_space}
\end{figure}

We next compare the computational costs of the two methods using the
approaches described in \Cref{sec:cost}. All runs employ $\mathbb{P}^2$
finite elements, $\Delta t=10^{-5}$, and final time $T=1$.
The measured CPU times and the speedups relative to Approach~I are reported
in \Cref{tab:cpu_stiff_heat}. Approach~III is the most efficient of the
three approaches at every mesh resolution. Its speedup increases from
$2.67\times$ for $N=40$ to $15.63\times$ for $N=80$ and
$38.09\times$ for $N=160$. In particular, for $N=160$, the CPU time
decreases from $24294.81$ seconds for Approach~I to $637.83$ seconds
for Approach~III. By comparison, Approach~II yields only modest speedups,
ranging from $1.02\times$ to $1.09\times$. These results demonstrate the
substantial computational advantage of Approach~III for this stiff test
problem, with the advantage becoming more pronounced as the mesh is refined.

\begin{table}[htbp]
\centering
\caption{CPU times and speedups for the stiff heat equation
\eqref{eq:stiff_heat} with periodic boundary conditions, using
$P_2$ finite elements. Speedups are measured relative to
Approach~I at each mesh resolution.}
\label{tab:cpu_stiff_heat}
\small
\setlength{\tabcolsep}{3.5pt}
\begin{adjustbox}{max width=\linewidth}
\begin{tabular}{@{}l rr rr rr@{}}
\toprule
\multirow{2}{*}{Approach}
& \multicolumn{2}{c}{$N=40$}
& \multicolumn{2}{c}{$N=80$}
& \multicolumn{2}{c}{$N=160$} \\
\cmidrule(lr){2-3} \cmidrule(lr){4-5} \cmidrule(l){6-7}
& Time (s) & Speedup
& Time (s) & Speedup
& Time (s) & Speedup \\
\midrule
I
& 598.47 & $1.00\times$
& 6597.74 & $1.00\times$
& 24294.81 & $1.00\times$ \\
II
& 553.10 & $1.08\times$
& 6444.95 & $1.02\times$
& 22359.75 & $1.09\times$ \\
\textbf{III}
& \textbf{223.95} & \textbf{$2.67\times$}
& \textbf{422.25} & \textbf{$15.63\times$}
& \textbf{637.83} & \textbf{$38.09\times$} \\
\bottomrule
\end{tabular}
\end{adjustbox}
\end{table}

We then investigate temporal convergence. We fix $N=80$, $k=2$, and
$r^0=10$, and successively halve the time step from
$\Delta t=10^{-2}$ to $\Delta t=6.25\times10^{-4}$. The $L^2(\Omega)$
error is evaluated at $T=1$. As shown in the left panel of
\Cref{fig:stiff_heat_time_rank}, both FR-MLFEM and ABUG-MLFEM
exhibit first-order convergence in time. Their errors are nearly identical
for all tested time steps, indicating that the augmented BUG integrator
preserves the temporal accuracy of the underlying time discretization in
these experiments.

Finally, we examine the rank evolution of ABUG-MLFEM for
polynomial degrees $k=1,2,3$ up to $T=2.5$, using $\Delta t=10^{-2}$
and truncation tolerance $\vartheta=10^{-12}$. The rank histories are shown
in the right panel of \Cref{fig:stiff_heat_time_rank}. Each term in
\eqref{eq:stiff_forcing} is separable, since
\(
  \exp\!\bigl(-j(x^2+y^2)\bigr)
  =\exp(-jx^2)\exp(-jy^2).
\)
Consequently, the forcing term has separation rank at most $11$.
The geometrically decreasing coefficients $10^{-(j-1)}$ produce modes
of progressively smaller amplitude. As these modes contribute to the
evolving solution, the adaptive integrator increases the rank in discrete
steps, from $8$ to $9$ and then to $10$. The rank subsequently remains
at $10$ over the simulated interval, with the remaining modes not retained
under the prescribed truncation tolerance. The rank histories are similar
for all three polynomial degrees, although the times at which the rank
increases differ: the $\mathbb{P}^2$ computation reaches rank $10$ at
approximately $t=0.2$, whereas the $\mathbb{P}^1$ and $\mathbb{P}^3$
computations reach this rank at approximately $t=0.6$. This behavior
suggests that the adaptive rank is primarily determined by the effective
low-rank structure of the evolving solution, while the spatial
discretization can affect the timing of rank enrichment.

\begin{figure}[htbp]
  \centering
  \begin{subfigure}[b]{0.42\textwidth}
    \centering
    \includegraphics[width=\linewidth,height=5cm,keepaspectratio]{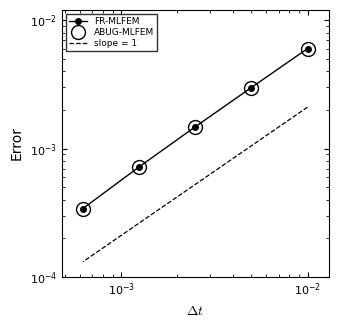}
  \end{subfigure}
  \hfill
  \begin{subfigure}[b]{0.56\textwidth}
    \centering
    \raisebox{0.14cm}{%
      \includegraphics[width=\linewidth,height=4.83cm,keepaspectratio]
      {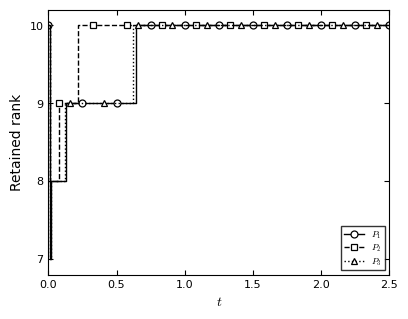}}
  \end{subfigure}
  \caption{Stiff heat equation \eqref{eq:stiff_heat}.
Left: temporal convergence at $T=1$ for FR-MLFEM and ABUG-MLFEM,
with $N=80$, $k=2$, and $r^0=10$.
Right: rank evolution of ABUG-MLFEM up to $T=2.5$ for
$k=1,2,3$, with $\Delta t=10^{-2}$ and truncation tolerance
$\vartheta=10^{-12}$.}
  \label{fig:stiff_heat_time_rank}
\end{figure}

\end{example}

%=================================================================================
\section{Conclusions}
\label{sec:conclusion}
%=================================================================================

We have presented a BUG-MLFEM coupled with fast, non-iterative Sylvester solvers for two-dimensional reaction-diffusion equations. Under the WDLRA framework, the scheme guarantees discrete well-posedness and mass-lumped $L^2$-stability while maintaining optimal $O(h^{k+1}+\Delta t)$ convergence across Dirichlet, Neumann, and periodic BCs. By exploiting the sparse-by-small structure of the intermediate $K$- and $L$-steps, our tailored solvers reduce the per-step computational cost to $O((m+n)r^2)$ and storage to $O((m+n)r)$. Numerical benchmarks confirm up to a $81\times$ speedup over full-rank and conventional low-rank approaches on mesh grid $N=160$, with performance gains becoming increasingly pronounced on refined meshes. Furthermore, the incorporation of the augmented BUG integrator confirms the superior robustness and adaptability of the proposed methodology.

The proposed low-cost non-iterative solvers can also be applied, either directly or with slight modifications, to the Sylvester equations arising from the BUG formulation of matrix ODEs obtained from finite difference discretizations.
Potential extensions of this work include error analysis, time-dependent diffusion coefficients, additional advection terms, three-dimensional problems, high-order time discretizations \cite{hauck2026high}, and general finite element methods without mass lumping, all of which will be pursued in future research.

\appendix
\section{GSVD}
We first construct the initial coefficient matrix $W^0$, and then obtain the initial matrices $U^0$, $S^0$, and $V^0$ by applying the generalized singular value decomposition (GSVD) to $W^0$.

We represent the discrete initial data as
\[
w_h(x,y,0)=\Phi^{\top}(x)W^0\Psi(y),
\qquad
W^0\in\mathbb{R}^{m\times n}.
\]
To determine $W^0$, we test the initial condition against
\[
z_h=\Phi^\top e_i(e'_j)^\top\Psi,
\qquad
i=1,\ldots,m,\quad j=1,\ldots,n,
\]
where $e_i\in\mathbb{R}^m$ and $e'_j\in\mathbb{R}^n$ are the canonical basis vectors. Using the mass-lumped inner product, we have
\[
(w_h(x,y,0),z_h)_h=(w_0,z_h)_h.
\]
Together with \eqref{mass_lumped_inner_product}, \eqref{eq:M_x_M_y}, and \Cref{lem:inner-products}, this relation gives
\[
W^0_{ij}=w_0(x_i,y_j),
\qquad
i=1,\ldots,m,\quad j=1,\ldots,n.
\]

Given $W^0 \in \mathbb{R}^{m\times n}$ and diagonal positive-definite mass
matrices $M_x \in \mathbb{R}^{m\times m}$ and $M_y \in \mathbb{R}^{n\times n}$, we seek a
rank-$r$ factorization
\[
    W^0 \approx U^0 S^0 (V^0)^\top,
\]
where the bases $U^0$ and $V^0$ satisfy
\[
    (U^0)^\top M_x U^0 = I_r,
    \qquad
    (V^0)^\top M_y V^0 = I_r.
\]
Since $M_x$ and $M_y$ are diagonal and positive definite, their square roots and
inverse square roots exist and are diagonal. We therefore introduce the weighted
transformation
\begin{equation}
    \widetilde{W}
    = M_x^{1/2} W^0 M_y^{1/2}.
    \label{eq:gwsvd-transform}
\end{equation}
Applying the standard truncated SVD to $\widetilde{W}$ gives
\begin{equation}
    \widetilde{W}
    \approx
    \widehat{U}_r \widehat{\Sigma}_r \widehat{V}_r^\top,
    \label{eq:gwsvd-svd}
\end{equation}
where
$\widehat{U}_r^\top\widehat{U}_r
=\widehat{V}_r^\top\widehat{V}_r=I_r$.
Transforming the factors back yields
\begin{equation}
    U^0 = M_x^{-1/2}\widehat{U}_r,
    \qquad
    V^0 = M_y^{-1/2}\widehat{V}_r,
    \qquad
    S^0 = \widehat{\Sigma}_r.
    \label{eq:gwsvd-undo}
\end{equation}
Indeed, the bases satisfy
\[
    (U^0)^\top M_x U^0
    = \widehat{U}_r^\top \widehat{U}_r
    = I_r,
    \qquad
    (V^0)^\top M_y V^0
    = \widehat{V}_r^\top \widehat{V}_r
    = I_r.
\]
Moreover,
\[
    U^0 S^0 (V^0)^\top
    =
    M_x^{-1/2}
    \widehat{U}_r
    \widehat{\Sigma}_r
    \widehat{V}_r^\top
    M_y^{-1/2},
\]
which provides the desired weighted rank-$r$ approximation of $W^0$. The implentation is given in \Cref{alg:gsvd}.

\begin{algorithm}[htbp]
\caption{GSVD.}
\label{alg:gsvd}
\begin{algorithmic}[1]

\REQUIRE
Matrix $W^0 \in \mathbb{R}^{m\times n}$; diagonal positive-definite mass matrices
$M_x  \in \mathbb{R}^{m\times m}$ and
$M_y  \in \mathbb{R}^{n\times n}$; rank $r$.

\ENSURE
$U^0 \in \mathbb{R}^{m\times r}$, $S^0 \in \mathbb{R}^{r\times r}$,
$V^0 \in \mathbb{R}^{n\times r}$ such that
$W^0 \approx U^0S^0(V^0)^\top$,
$(U^0)^\top M_xU^0=I_r$, and
$(V^0)^\top M_yV^0=I_r$.

\STATE Form
$M_x^{1/2}$ and
$M_y^{1/2}$.
\hfill $\triangleright$ elementwise

\STATE $\widetilde{W}=
M_x^{1/2}W^0M_y^{1/2}$.

\STATE $(\widehat U_r,S^0,\widehat V_r^\top)
= \operatorname{truncatedSVD}(\widetilde{W},r)$.

\STATE $U^0= M_x^{-1/2}\widehat U_r$,
\quad
$V^0= M_y^{-1/2}\widehat V_r$.
\hfill $\triangleright$ elementwise scaling

\STATE \RETURN $U^0,S^0,V^0$

\end{algorithmic}
\end{algorithm}

\section{Augmented BUG integrator}\label{abug}
We extend \Cref{alg:KLS-Steps} by incorporating the augmented BUG integrator \cite{ceruti2022rank}, yielding the augmented BUG-MLFEM (ABUG-MLFEM).
The resulting algorithm is presented in \Cref{alg:AugBUG-Steps}.

\begin{algorithm}[htbp]
\caption{One-step augmented BUG integrator (ABUG-MLFEM).}
\label{alg:AugBUG-Steps}
\begin{algorithmic}[1]

\REQUIRE
$U^\mathfrak{n} \in \mathbb{R}^{m\times r^\mathfrak{n}},\, S^\mathfrak{n} \in \mathbb{R}^{r^\mathfrak{n}\times r^\mathfrak{n}},\, V^\mathfrak{n} \in \mathbb{R}^{n\times r^\mathfrak{n}},\, \Delta t,\, M_x,\, M_y$, tolerance $\vartheta$.
\ENSURE
$U^{\mathfrak{n}+1},\, S^{\mathfrak{n}+1},\, V^{\mathfrak{n}+1}$ of rank $r^{\mathfrak{n}+1}$.

\algcomment{\textbf{Step 1}: Augment $U^\mathfrak{n} \to \widehat{U} \in \mathbb{R}^{m\times \hat{r}}$ and $V^\mathfrak{n} \to \widehat{V} \in \mathbb{R}^{n\times \hat{r}}$, $\hat{r} \le 2r^\mathfrak{n}$, in parallel:}
\STATE \textit{K-step:} Fix $V^\mathfrak{n}$, and solve $K^{\mathfrak{n}+1} \in \mathbb{R}^{m \times r^\mathfrak{n}}$ from
\begin{equation}
  D_t K^{\mathfrak{n}+1}
  = \mathscr{L}(K^{\mathfrak{n}+1} (V^\mathfrak{n})^\top)M_y V^\mathfrak{n},
  \qquad K^\mathfrak{n} = U^\mathfrak{n} S^\mathfrak{n}.
  \label{eq:aug_k_eqn}
\end{equation}
\STATE Compute $[\widehat{U},\, R_K] = \mathrm{GQR}\!\left([\,U^\mathfrak{n},\, K^{\mathfrak{n}+1}\,],\, M_x\right)$.
\STATE Compute the $\hat{r} \times r^\mathfrak{n}$ matrix $\widehat{\mathcal{M}} = \widehat{U}^\top M_x U^\mathfrak{n}$.
\STATE \textit{L-step:} Fix $U^\mathfrak{n}$, and solve $L^{\mathfrak{n}+1} \in \mathbb{R}^{n \times r^\mathfrak{n}}$ from
\begin{equation}
  D_t L^{\mathfrak{n}+1}
  = \mathscr{L}(U^\mathfrak{n} (L^{\mathfrak{n}+1})^\top)^\top M_x U^\mathfrak{n},
  \qquad L^\mathfrak{n} = V^\mathfrak{n} (S^\mathfrak{n})^\top.
  \label{eq:aug_l_eqn}
\end{equation}
\STATE Compute $[\widehat{V},\, R_L] = \mathrm{GQR}\!\left([\,V^\mathfrak{n},\, L^{\mathfrak{n}+1}\,],\, M_y\right)$.
\STATE Compute the $\hat{r} \times r^\mathfrak{n}$ matrix $\widehat{\mathcal{N}} = \widehat{V}^\top M_y V^\mathfrak{n}$.

\algcomment{\textbf{Step 2}: Galerkin update $S^\mathfrak{n} \to \widehat{S}^{\mathfrak{n}+1} \in \mathbb{R}^{\hat{r}\times\hat{r}}$ in the augmented space:}
\STATE \textit{S-step:} Compute the projected initial matrix $\widehat{S}^{\mathfrak{n},*} = \widehat{\mathcal{M}}\, S^\mathfrak{n}\, \widehat{\mathcal{N}}^\top \in \mathbb{R}^{\hat{r}\times\hat{r}}$.
\STATE Solve $\widehat{S}^{\mathfrak{n}+1} \in \mathbb{R}^{\hat{r} \times \hat{r}}$ from
\begin{equation}\label{eq:aug_s_eqn}
  \frac{\widehat{S}^{\mathfrak{n}+1} - \widehat{S}^{\mathfrak{n},*}}{\Delta t}
  = \widehat{U}^\top M_x\, \mathscr{L}\!\left(\widehat{U}\, \widehat{S}^{\mathfrak{n}+1}\, \widehat{V}^\top\right) M_y \widehat{V}.
\end{equation}

\algcomment{\textbf{Step 3}: Truncate $\widehat{U}\,\widehat{S}^{\mathfrak{n}+1}\,\widehat{V}^\top$ to rank $r^{\mathfrak{n}+1} \le \hat{r}$ with tolerance $\vartheta$:}
\STATE Compute the SVD 
\begin{equation}
\widehat S^{\mathfrak n+1}
=
P\Sigma Q^\top,
\qquad
\Sigma=\operatorname{diag}
(\sigma_1,\ldots,\sigma_{\hat r}),
\qquad
\sigma_1\ge\cdots\ge\sigma_{\hat r}\ge0.
\end{equation}
\STATE Determine the smallest $r^{\mathfrak n+1}\le\hat r$ satisfying $\left(
\sum_{j=r^{\mathfrak n+1}+1}^{\hat r}
\sigma_j^2
\right)^{1/2}
\le\vartheta.$
% \begin{equation}\label{eq:aug_rank}
%   r^{\mathfrak{n}+1} = \min\Big\{ k \in \{1,\ldots,\hat{r}\} \;:\; \Big(\textstyle\sum_{j=k+1}^{\hat{r}} \sigma_j^2\Big)^{1/2} \le \vartheta \Big\}.
% \end{equation}
\STATE Set $U^{\mathfrak{n}+1} = \widehat{U} P_{:,1:r^{\mathfrak{n}+1}}$, $\;S^{\mathfrak{n}+1} = \mathrm{diag}(\sigma_1, \ldots, \sigma_{r^{\mathfrak{n}+1}})$, $\;V^{\mathfrak{n}+1} = \widehat{V} Q_{:,1:r^{\mathfrak{n}+1}}$.
\end{algorithmic}
\end{algorithm}

\section*{Acknowledgments}
This research was supported by the University of Texas at El Paso Startup Award.

\bibliographystyle{siamplain}
\bibliography{references}
\end{document}